\documentclass{amsart}

\usepackage{amsmath, amssymb, amsthm}
\usepackage{bbm}

\newtheorem{theorem}{Theorem}[section]
\newtheorem{lemma}[theorem]{Lemma}
\newtheorem{proposition}[theorem]{Proposition}
\newtheorem{corollary}[theorem]{Corollary}
\newtheorem{claim}[theorem]{Claim}
\newtheorem{subclaim}[theorem]{Subclaim}
\newtheorem{fact}[theorem]{Fact}
\newtheorem{question}[theorem]{Question}

\theoremstyle{definition}
\newtheorem{definition}[theorem]{Definition}
\newtheorem{remark}[theorem]{Remark}

\newcommand{\cf}{\mathrm{cf}}

\newcommand{\bb}{\mathbb}
\newcommand{\bbm}{\mathbbm}

\begin{document}
\title{Squares and narrow systems}
\author{Chris Lambie-Hanson}
\address{Department of Mathematics, Bar-Ilan University \\ Ramat Gan, 5290002, Israel}
\email{lambiec@macs.biu.ac.il}
\thanks{This research was undertaken while the author was a Lady Davis Postdoctoral Fellow. 
The author would like to thank the Lady Davis Fellowship Trust and the Hebrew University of Jerusalem. 
The author would also like to thank Menachem Magidor for many helpful discussions and the 
anonymous referee for a number of useful suggestions.}
\begin{abstract}
	A narrow system is a combinatorial object introduced by Magidor and Shelah in connection with work on the tree property at successors of singular cardinals. In analogy to the tree property, a cardinal $\kappa$ satisfies the \emph{narrow system property} if every narrow system of height $\kappa$ has a cofinal branch. In this paper, we study connections between the narrow system property, square principles, and forcing axioms. We prove, assuming large cardinals, both that it is consistent that $\aleph_{\omega+1}$ satisfies the narrow system property and $\square_{\aleph_{\omega}, < \aleph_{\omega}}$ holds and that it is consistent that every regular cardinal satisfies the narrow system property. We introduce natural strengthenings of classical square principles and show how they can be used to produce narrow systems with no cofinal branch. Finally, we show that the Proper Forcing Axiom implies that every narrow system of countable width has a cofinal branch but is consistent with the existence of a narrow system of width $\omega_1$ with no cofinal branch.
\end{abstract}
\maketitle

\section{Introduction}

The question as to when certain large cardinal properties can hold at accessible cardinals has been of considerable interest in modern set theory. Of particular interest are successors of singular cardinals (particularly $\aleph_{\omega + 1}$), at which these properties are typically more difficult and require larger cardinals to attain than at successors of regular cardinals. One of the large cardinal properties that has received a great deal of attention is the tree property. In \cite{magidorshelah}, Magidor and Shelah prove that the tree property holds at the successor of a singular limit of strongly compact cardinals and that, assuming large cardinals (roughly a huge cardinal with $\omega$-many supercompact cardinals above it), it is consistent that the tree property holds at $\aleph_{\omega + 1}$. In the same paper, they introduce the notion of a narrow system, which has proved to be a valuable tool in the analysis of the tree property at successors of singular cardinals and is the primary subject of this paper. In \cite{sinapova}, Sinapova reduces the large cardinals needed to obtain the tree property at $\aleph_{\omega + 1}$ by forcing it from $\omega$-many supercompact cardinals. In \cite{neeman}, Neeman shows the consistency of the tree property holding simultaneously at $\aleph_{\omega + 1}$ and $\aleph_{n}$ for all $2 \leq n < \omega$ and, in the process, demonstrates a different method for forcing the tree property at $\aleph_{\omega + 1}$ from $\omega$-many supercompact cardinals.

The forcing constructions employed by Magidor and Shelah, Sinapova, and Neeman are all distinct, but, in all known models in which the tree property holds at the successor of a singular cardinal, $\mu$, the verification of the tree property follows the same general two-step pattern. In the first step, it is argued that every $\mu^+$-tree admits a narrow system of height $\mu^+$ (a precise definition of this will be given later). In the second step, it is argued that every narrow system of height $\mu^+$ has a cofinal branch. With this in mind and with an eye towards getting a better understanding of matters surrounding the tree property, we focus our attention here on these two steps, and in particular on the latter, taken individually.

The general structure of the paper is as follows. In Section \ref{narrowSystemSect}, we introduce narrow systems and prove some basic facts about them. In Section \ref{forcingSect}, we recall some combinatorial and forcing notions that will be useful throughout the paper. In Section \ref{preservationSect}, we present branch preservation lemmas for narrow systems. These are slight improvements on a similar lemma of Sinapova from \cite{sinapova}. In Section \ref{weakSquareSect}, we present some forcing constructions related to the narrow system property. In particular, starting from large cardinals, we obtain a model in which every narrow system has a cofinal branch, a model in which the narrow system property at $\aleph_{\omega + 1}$ and $\square_{\aleph_\omega, < \aleph_\omega}$ both hold, and a model in which there is an inaccessible, non-weakly compact $\lambda$ such that the narrow system property holds at $\lambda$. In Section \ref{counterexampleSect}, we demonstrate how branchless narrow systems can be constructed from certain strengthenings of classical square principles. In Section \ref{indexedSquareSect}, we demonstrate how to force some of these square principles and, in Section \ref{separatingSect}, we demonstrate how to separate certain of them from one another. In Section \ref{derivedSect}, we discuss derived systems and use them to get finer control over the failure of the narrow system property. In Section \ref{pfaSect}, we show that the Proper Forcing Axiom implies that every narrow system with countable width has a cofinal branch but has no effect on narrow systems with uncountable width. At the end, we present some open questions.

Our notation is for the most part standard. The reference for all undefined notions and notations is \cite{jech}. If $A$ is a set of ordinals, then $\mathrm{otp}(A)$ denotes the order type of $A$ and $A'$ denotes the set of limit points of $A$, i.e. $\{\alpha \in A \mid \sup(A \cap \alpha) = \alpha\}$. If $X$ is a set and $\kappa$ is a cardinal, then $[X]^\kappa = \{Y \subseteq X \mid |Y| = \kappa \}$. If $\kappa < \lambda$ are regular cardinals, then $S^\lambda_\kappa = \{\alpha < \lambda \mid \cf(\alpha) = \kappa \}$, and $\mathrm{cof}(\kappa)$ denotes the class of ordinals of cofinality $\kappa$. $\mathrm{On}$ denotes the class of all ordinals. If $\lambda$ is an uncountable, regular cardinal, $T$ is a $\lambda$-tree, and $\alpha < \lambda$, then we will assume that level $\alpha$ of $T$ is $\{\alpha\} \times \kappa_\alpha$, where $\kappa_\alpha < \lambda$. In particular, if $\lambda = \mu^+$, then, for all $0 < \alpha < \lambda$, we will assume that level $\alpha$ of $T$ is $\{\alpha\} \times \mu$. The tree relation for a tree $T$ will typically be denoted by $<_T$. If $R$ is a binary relation, then we will typically write $a <_R b$ to stand for $(a, b) \in R$. 
\section{Narrow systems} \label{narrowSystemSect}

\begin{definition}
	Let $R$ be a binary relation on a set $X$.
	\begin{itemize}
		\item{If $a,b \in X$, then $a$ and $b$ are \emph{$R$-comparable} if $a = b$, $a <_R b$, or $b <_R a$. Otherwise, $a$ and $b$ are \emph{$R$-incomparable}, which is denoted $a \perp_R b$.}
		\item{$R$ is \emph{tree-like} if, for all $a,b,c \in X$, if $a <_R c$ and $b <_R c$, then $a$ and $b$ are $R$-comparable.}
	\end{itemize}
\end{definition}

We now recall the notion of a $\lambda$-system, introduced in \cite{magidorshelah}.

\begin{definition} \label{system_def}
  Let $\lambda$ be an infinite, regular cardinal. $S = \langle \bigcup_{\alpha \in I}\{\alpha\} \times \kappa_\alpha, \mathcal{R} \rangle$ is a \emph{$\lambda$-system} if:
	\begin{enumerate}
		\item{$I \subseteq \lambda$ is unbounded and, for all $\alpha \in I$, $\kappa_\alpha$ is a cardinal such that $0 < \kappa_\alpha < \lambda$. We sometimes slightly 
        abuse notation and write $S$ to denote its underlying set, i.e. $\bigcup_{\alpha \in I}\{\alpha\} \times \kappa_\alpha$. For example, 
        we will write $u \in S$ instead of $u \in \bigcup_{\alpha \in I}\{\alpha\} \times \kappa_\alpha$.
        For each $\alpha \in I$, we say that $S_\alpha := \{\alpha\} \times \kappa_\alpha$ is the $\alpha^{\mathrm{th}}$ level of $S$. Similarly, 
        if $\vartriangleleft$ is one of $<, \leq, >,$ or $\geq$ and $\beta < \lambda$, then $S_{\vartriangleleft \beta}:= 
        \bigcup\{\{\alpha\} \times \kappa_\alpha \mid \alpha \in I$ and $\alpha \vartriangleleft \beta\}$.}
		\item{$\mathcal{R}$ is a set of binary, transitive, tree-like relations on $S$ and $0 < |\mathcal{R}| < \lambda$;}
		\item{for all $R \in \mathcal{R}$, $\alpha_0, \alpha_1 \in I$, $\beta_0 < \kappa_{\alpha_0}$, and $\beta_1 < \kappa_{\alpha_1}$, if $(\alpha_0, \beta_0) <_R (\alpha_1,\beta_1)$, then $\alpha_0 < \alpha_1$;}
		\item{for all $\alpha_0 < \alpha_1$, both in $I$, there are $\beta_0 < \kappa_{\alpha_0}$, $\beta_1 < \kappa_{\alpha_1}$, and $R \in \mathcal{R}$ such that $(\alpha_0, \beta_0) <_R (\alpha_1, \beta_1)$.}
	\end{enumerate}

	If $S = \langle \bigcup_{\alpha \in I}\{\alpha\} \times \kappa_\alpha, \mathcal{R} \rangle$ is a $\lambda$-system, then we define 
  $\mathrm{width}(S) = \max(\sup(\{\kappa_\alpha \mid \alpha \in I\}), |\mathcal{R}|)$ and $\mathrm{height}(S) = \lambda$. $S$ is a \emph{narrow $\lambda$-system} if $\mathrm{width}(S)^+ < \lambda$.

	$S$ is a \emph{strong $\lambda$-system} if it satisfies the following strengthening of (4):
	\begin{enumerate}
    \item[($4'$)]{for all $\alpha_0 < \alpha_1$, both in $I$, and for every $\beta_1 < \kappa_{\alpha_1}$, there are $\beta_0 < \kappa_{\alpha_0}$ and $R \in \mathcal{R}$ such that $(\alpha_0, \beta_0) <_R (\alpha_1,\beta_1)$.}
	\end{enumerate}

	If $R \in \mathcal{R}$, a \emph{branch of $S$ through $R$} is a set $b \subseteq S$ such that for all $u,v \in b$, $u$ and $v$ are $R$-comparable. $b$ is a \emph{cofinal branch} if, for unboundedly many $\alpha \in I$, $b \cap S_\alpha \not= \emptyset$.
\end{definition}

\begin{remark}
  If $\lambda$ is a successor cardinal and $S = \langle \bigcup_{\alpha \in I}\{\alpha\} \times \kappa_\alpha, \mathcal{R} \rangle$ is a $\lambda$-system, or if $\kappa$ is weakly inaccessible and $S$ is a narrow $\lambda$-system, then there is an unbounded $J \subseteq I$ and a $\kappa < \lambda$ such that, for all $\alpha \in J$, $\kappa_\alpha = \kappa$. It will then be sufficient for us to work with subsystems of the form $\langle J \times \kappa, \mathcal{R} \rangle$, so, in the case that $\lambda$ is a successor cardinal or we are considering only narrow systems, we will assume our systems are of this form. In addition, if $S$ is a $\lambda$-system, then we will write $\mathcal{R}(S)$ to refer to the set of relations of $S$.
\end{remark}

\begin{definition}
	Let $\lambda$ be a regular cardinal, and let $T$ be a $\lambda$-tree. $T$ \emph{admits a narrow system} if there is an unbounded $I \subseteq \lambda$ and a $\kappa$ with $\kappa^+ < \lambda$ such that $\langle I \times \kappa, \{<_T\} \rangle$ is a system.
\end{definition}

\begin{remark}
	Note that, as $T$ is a tree, verifying that $\langle I \times \kappa, \{<_T\} \rangle$ in the above definition is a system 
  amounts to checking condition (4) in Definition \ref{system_def}.
\end{remark}

We will be interested in statements asserting that all narrow systems of a certain shape have a cofinal branch. 
We first show that, when verifying that all narrow systems of a given height and width have a cofinal branch, 
it suffices to consider systems having a single relation.

\begin{proposition} \label{singlerelation}
	Let $\lambda$ be an uncountable, regular cardinal, and suppose $S = \langle I \times \kappa, \mathcal{R} \rangle$ is a $\lambda$-system with no cofinal branch. 
  Suppose $\mathrm{width}(S) = \kappa'$. Then there is a $\lambda$-system $S' = \langle I \times \kappa', \mathcal{R}' \rangle$ 
  with no cofinal branch such that $|\mathcal{R}'| = 1$.
\end{proposition}

\begin{proof}
	Let $\kappa_0 = |\mathcal{R}|$, and enumerate $\mathcal{R}$ as $\langle R_\xi \mid \xi < \kappa_0 \rangle$. 
  Fix a bijection $\pi : \kappa' \rightarrow \kappa \times \kappa_0$. For $\beta < \kappa'$, denote $\pi(\beta)$ 
  as $(\beta^0, \beta^1)$. Let $\mathcal{R}' = \{R\}$, and define the system $S'$ by letting 
  $(\alpha_0, \beta_0) <_R (\alpha_1, \beta_1)$ iff $\beta^1_0 = \beta^1_1$ and, in $S$, 
  $(\alpha_0, \beta^0_0) <_{R_{\beta^1_0}} (\alpha_1, \beta^0_1)$. It is easily verified that $S'$ is a 
  $\lambda$-system and that a cofinal branch through $S'$ would give rise to a cofinal branch through $S$.
\end{proof}

\begin{definition}
	Let $\kappa \leq \lambda$ be infinite cardinals. The \emph{$(\kappa, \lambda)$-narrow system property} (abbreviated $NSP(\kappa, \lambda)$) holds if every narrow system of width $<\kappa$ and height $\lambda$ has a cofinal branch. The \emph{$(\kappa, \geq \lambda)$-narrow system property} ($NSP(\kappa, \geq \lambda)$) holds if every narrow system of width $<\kappa$ and height a regular cardinal $\geq \lambda$ has a cofinal branch.

	If $\lambda$ is a regular, uncountable cardinal, then the \emph{narrow system property} holds at $\lambda$ (abbreviated $NSP(\lambda)$) if every narrow system of height $\lambda$ has a cofinal branch. Note that this is the same as $NSP(\lambda, \lambda)$.
\end{definition}

The reader may be wondering why we are focusing on narrow systems, for which we require $\mathrm{width}(S)^+ < \mathrm{height}(S)$, 
rather than adopting the seemingly more natural requirement of $\mathrm{width}(S) < \mathrm{height}(S)$. One reason for this is that 
the analogue of the narrow system property for systems satisfying $\mathrm{width}(S) < \mathrm{height}(S)$ is inconsistent: as will be 
shown in Proposition \ref{wide_system_prop}, it is a theorem of ZFC that, for every infinite cardinal $\kappa$, there is a $\kappa^+$-system 
of width $\kappa$ with no cofinal branch.

\begin{remark} \label{singlerelationremark}
  Note that, by Proposition \ref{singlerelation}, if $\kappa \leq \lambda$ are infinite cardinals and there is a counterexample to $NSP(\kappa, \lambda)$, 
  then there is a counterexample $S$ such that $|\mathcal{R}(S)| = 1$. Therefore, in order to verify $NSP(\kappa, \lambda)$, it suffices to verify that 
  every narrow $\lambda$-system $S$ such that $\mathrm{width}(S)<\kappa$ and $|\mathcal{R}(S)| = 1$ has a cofinal branch.
\end{remark}

\begin{proposition} \label{wcprop}
	If $\lambda$ is weakly compact, then $NSP(\lambda)$ holds.
\end{proposition}

\begin{proof}
	Suppose $\lambda$ is weakly compact, $\kappa < \lambda$, and $S = \langle I \times \kappa, \mathcal{R} \rangle$ is a $\lambda$-system. We will show that $S$ has a cofinal branch. 
  By Remark \ref{singlerelationremark}, we may assume that $\mathcal{R} = \{R\}$. We define a function $f:[I]^2 \rightarrow \kappa \times \kappa$ as follows. For every $\alpha_0 < \alpha_1$ with $\alpha_0, \alpha_1 \in I$, find $\beta, \gamma \in \kappa$ such that $(\alpha_0, \beta) <_R (\alpha_1, \gamma)$ and let $f(\{\alpha_0, \alpha_1\}) = (\beta, \gamma)$. Since $\lambda$ is weakly compact, $\lambda \rightarrow (\lambda)^2_\kappa$, so there are an unbounded $H \subseteq I$ and $\beta^*, \gamma^* < \kappa$ such that, for all $\alpha_0 < \alpha_1$ with $\alpha_0, \alpha_1 \in H$, $f(\{\alpha_0, \alpha_1\}) = (\beta^*, \gamma^*)$.

	Now let $\alpha_0 < \alpha_1 < \alpha_2$, with all three in $H$. $(\alpha_0, \beta^*) <_R (\alpha_2, \gamma^*)$ and $(\alpha_1, \beta^*) <_R (\alpha_2, \gamma^*)$, so, since $R$ is tree-like, $(\alpha_0, \beta^*) <_R (\alpha_1, \beta^*)$. Thus, $\{(\alpha, \beta^*) \mid \alpha \in H\}$ is a cofinal branch through $S$.
\end{proof}

\begin{proposition}
	If $\lambda$ is strongly compact, then $NSP(\lambda, \geq \lambda)$ holds.
\end{proposition}

\begin{proof}
	Suppose $\lambda$ is strongly compact, $\kappa < \lambda$, $\mu \geq \lambda$ is a regular cardinal, and $S = \langle I \times \kappa, \mathcal{R} \rangle$ is a $\mu$-system. We may again assume that $\mathcal{R} = \{R\}$. Since $\lambda$ is strongly compact, every $\lambda$-complete filter can be extended to a $\lambda$-complete ultrafilter. Thus, let $U$ be a $\lambda$-complete ultrafilter over $\mu$ containing $I$ and all co-bounded subsets of $\mu$. As in the proof of Proposition \ref{wcprop}, define $f:[I]^2 \rightarrow \kappa \times \kappa$ so that, if $\alpha_0 < \alpha_1$ are in $I$ and $f(\{\alpha_0, \alpha_1\}) = (\beta, \gamma)$, then $(\alpha_0, \beta) <_R (\alpha_1, \gamma)$.

	For each $\alpha \in I$, use the $\lambda$-completeness of $U$ to find $\beta_\alpha, \gamma_\alpha < \kappa$ such that $A_\alpha := \{\alpha' \in I \setminus (\alpha + 1) \mid f(\{\alpha, \alpha'\}) = (\beta_\alpha, \gamma_\alpha)\} \in U$. Again using the $\lambda$-completeness of $U$, find $B \subseteq I$ and $\beta^*, \gamma^* < \kappa$ such that $B \in U$ and, for all $\alpha \in B$, $(\beta_\alpha, \gamma_\alpha) = (\beta^*, \gamma^*)$. Let $\alpha_0 < \alpha_1$ be in $B$, and let $\alpha_2 \in A_{\alpha_0} \cap A_{\alpha_1}$. Then $(\alpha_0, \beta^*) <_R (\alpha_2, \gamma^*)$ and $(\alpha_1, \beta^*) <_R (\alpha_2, \gamma^*)$, so, since $R$ is tree-like, $(\alpha_0, \beta^*) <_R (\alpha_1, \beta^*)$. Thus, $\{(\alpha, \beta^*) \mid \alpha \in B\}$ is a cofinal branch through $S$.
\end{proof}

\section{Combinatorial and forcing preliminaries} \label{forcingSect}

In this section, we recall some relevant combinatorial and forcing notions and basic facts thereon. 
We start with variations on Jensen's square principle, which will be important throughout this paper.

\begin{definition} \label{jensen_square_def}
	Let $\lambda$ and $\mu$ be cardinals, with $\mu$ infinite and $\lambda > 1$. A $\square_{\mu, <\lambda}$-sequence is a sequence $\vec{\mathcal{C}} = \langle \mathcal{C}_\alpha \mid \alpha < \mu^+ \rangle$ such that:
	\begin{enumerate}
	  \item{For all limit $\alpha < \mu^+$, $\mathcal{C}_\alpha$ is a collection of club subsets of $\alpha$ and $1 \leq |\mathcal{C}_\alpha| < \lambda$;}
		\item{for all limit $\alpha < \beta < \mu^+$ and all $C \in \mathcal{C}_\beta$, if $\alpha \in C'$, then $C \cap \alpha \in \mathcal{C}_\alpha$;}
    \item for all limit $\alpha < \mu^+$ and all $C \in \mathcal{C}_\alpha$, $\mathrm{otp(C)} \leq \mu$.
	\end{enumerate}
	$\square_{\mu, < \lambda}$ holds if there is a $\square_{\mu, < \lambda}$-sequence.
\end{definition}

\begin{remark} \label{square_remark}
  $\square_{\mu, < \lambda^+}$ is usually denoted $\square_{\mu, \lambda}$. It is immediate that, if $\lambda_0 < \lambda_1$, then $\square_{\mu, \lambda_0}$ implies $\square_{\mu, \lambda_1}$. $\square_{\mu,1}$ is Jensen's classical principle $\square_\mu$. $\square_{\mu, \mu}$ is also called \emph{weak square} and denoted $\square^*_\mu$. Jensen proved that $\square^*_\mu$ is equivalent to the existence of a special $\mu^+$-Aronszajn tree (for a proof of this fact, see \cite[Section 5]{cummings_notes}. $\square_{\mu, \mu^+}$ is also called \emph{silly square} and holds in all models of ZFC. 
\end{remark}

\begin{definition}
  Suppose $1 < \lambda \leq \kappa$ are cardinals, with $\kappa$ infinite and regular, and $\vec{\mathcal{C}} = \langle \mathcal{C}_\alpha 
  \mid \alpha < \kappa \rangle$ satisfies (1) and (2) of Definition \ref{jensen_square_def} for all $\alpha < \beta < \kappa$ ($\kappa$ 
  here is replacing the $\mu^+$ from Definition \ref{jensen_square_def}). A club 
  $D \subseteq \kappa$ is said to be a \emph{thread} through $\vec{\mathcal{C}}$ if, for all $\alpha \in D'$, $D \cap \alpha \in \mathcal{C}_\alpha$.
\end{definition}

Clause (3) in Definition \ref{jensen_square_def} easily implies that, if $\vec{\mathcal{C}}$ is a $\square_{\mu, < \lambda}$-sequence, then there 
is no thread through $\vec{\mathcal{C}}$. If we weaken clause (3) to just require its anti-thread consequence, then we obtain the 
definition of $\square(\kappa, < \lambda)$.

\begin{definition}
	Let $1 < \lambda \leq \kappa$ be cardinals, with $\kappa$ infinite and regular. A $\square(\kappa, < \lambda)$-sequence is a sequence $\vec{\mathcal{C}} = \langle \mathcal{C}_\alpha \mid \alpha < \kappa \rangle$ such that:
	\begin{enumerate}
		\item{for all limit $\alpha < \kappa$, $\mathcal{C}_\alpha$ is a collection of club subsets of $\alpha$ 
      and $1 \leq |\mathcal{C}_\alpha| < \lambda$;}
		\item{for all limit $\alpha < \beta < \kappa$ and all $C \in \mathcal{C}_\beta$, if $\alpha \in C'$, then $C \cap \alpha \in \mathcal{C}_\alpha$;}
    \item{there is no thread through $\vec{\mathcal{C}}$.}
	\end{enumerate}
	$\square(\kappa, < \lambda)$ holds if there is a $\square(\kappa, < \lambda)$-sequence. As above, we denote $\square(\kappa, < \lambda^+)$ by $\square(\kappa, \lambda)$ and $\square(\kappa, 1)$ by $\square(\kappa)$.
\end{definition}

We will also need the notion of approachability, which plays an important role in the study of successors of singular cardinals.

\begin{definition}
	Let $\kappa$ be a regular, uncountable cardinal, and let $\vec{a} = \langle a_\alpha \mid \alpha < \kappa \rangle$ 
  be a sequence of bounded subsets of $\kappa$. If $\gamma < \kappa$, then $\gamma$ is \emph{approachable with 
  respect to $\vec{a}$} if there is an unbounded $A \subseteq \gamma$ such that $\mathrm{otp}(A) = \cf(\gamma)$ and, 
  for every $\beta < \gamma$, there is $\alpha < \gamma$ such that $A \cap \beta = a_\alpha$.
\end{definition}

\begin{definition}
	Let $\kappa$ be a regular, uncountable cardinal.
	\begin{itemize}
		\item{If $A \subseteq \kappa$, then $A \in I[\kappa]$ if there is a club 
        $C \subseteq \kappa$ and a sequence $\vec{a} = \langle a_\alpha \mid \alpha < \kappa \rangle$ 
        of bounded subsets of $\kappa$ such that, for every $\gamma \in A \cap C$, $\cf(\gamma) < \gamma$ and 
        $\gamma$ is approachable with respect to $\vec{a}$.}
		\item{If $\mu$ is a singular cardinal, then $AP_\mu$ (the \emph{approachability property} at $\mu$) is the assertion that $\mu^+ \in I[\mu^+]$.}
	\end{itemize}
\end{definition}

$I[\kappa]$ is called the \emph{approachability ideal}. A wealth of information on $I[\kappa]$ can be found in \cite{eisworth}. We collect some of the relevant facts here.

\begin{remark}
	Let $\kappa$ be a regular, uncountable cardinal.
	\begin{itemize}
		\item{$I[\kappa]$ is a normal, $\kappa$-complete ideal extending the non-stationary ideal.}
		\item{Suppose $\kappa^{<\kappa} = \kappa$ and $\vec{a} = \langle a_\alpha \mid \alpha < \kappa \rangle$ 
      is a fixed enumeration of all bounded subsets of $\kappa$. Then $A \in I[\kappa]$ iff there is a club 
      $C \subseteq \kappa$ such that, for all $\gamma \in A \cap C$, $\cf(\gamma) < \gamma$ and $\gamma$ is approachable with respect to $\vec{a}$.}
		\item{If $\mu$ is an infinite cardinal, then $\square^*_\mu \Rightarrow AP_\mu$.}
		\item{If $\lambda$ is a supercompact cardinal and $\mu$ is a singular cardinal such that $\cf(\mu) < \lambda < \mu$, then $\neg AP_\mu$.}
	\end{itemize}
\end{remark}

We now move to forcing. We first recall the notion of strategic closure.

\begin{definition}
	Let $\mathbb{P}$ be a partial order and let $\beta$ be an ordinal.
	\begin{enumerate}
		\item {The two-player game $G_\beta(\mathbb{P})$ is defined as follows: Players I and II alternately play entries in $\langle p_\alpha \mid \alpha < \beta \rangle$, a decreasing sequence of conditions in $\mathbb{P}$ with $p_0 = \bbm{1}_{\mathbb{P}}$. Player I plays at odd stages, and Player II plays at even stages (including all limit stages). If there is an even stage $\alpha < \beta$ at which Player II cannot play, then Player I wins. Otherwise, Player II wins.}
    \item{$\mathbb{P}$ is \emph{$\beta$-strategically closed} if Player II has a winning strategy for the game $G_\beta(\mathbb{P})$. $\mathbb{P}$ 
      is \emph{$<\beta$-strategically closed} if it is $\alpha$-strategically closed for all $\alpha < \beta$.}
	\end{enumerate}
\end{definition}

We now introduce the standard forcing poset to add a $\square_{\mu, \lambda}$-sequence (cf. \cite{cfm}, Definition 6.2).

\begin{definition}
	Let $\lambda$ and $\mu$ be cardinals, with $1 \leq \lambda \leq \mu$ and $\mu$ uncountable. $\bb{S}(\mu, \lambda)$ is the forcing poset consisting of all conditions $p = \langle \mathcal{C}^p_\alpha \mid \alpha \leq \gamma^p \rangle$ such that:
	\begin{enumerate}
		\item{$\gamma^p < \mu^+$;}
		\item{for all limit $\alpha \leq \gamma^p$ and all $C \in \mathcal{C}^p_\alpha$, $C$ is a club in $\alpha$ and $\mathrm{otp}(C) \leq \mu$;}
		\item{for all limit $\alpha \leq \gamma^p$, $1 \leq |\mathcal{C}^p_\alpha| \leq \lambda$;}
		\item{for all limit $\alpha < \beta \leq \gamma^p$ and all $C \in \mathcal{C}^p_\beta$, if $\alpha \in C'$, then $C' \cap \alpha \in \mathcal{C}^p_\alpha$.}
	\end{enumerate}
	For all $p,q \in \bb{S}(\mu, \lambda)$, $q \leq p$ iff $q$ end-extends $p$, i.e., $\gamma^q \geq \gamma^p$ and, for all $\alpha \leq \gamma^p$, $\mathcal{C}^q_\alpha = \mathcal{C}^p_\alpha$.

	$\bb{S}(\mu, < \lambda)$ is defined similarly, except that, in item (3), we require $1 \leq |\mathcal{C}^p_\alpha| < \lambda$. 
  When considering $\bb{S}(\mu, < \lambda)$, we will assume $1 < \lambda \leq \mu^+$, as $\bb{S}(\mu, \mu) = \bb{S}(\mu, < \mu^+)$.
\end{definition}

Proofs of the following can be found in \cite{cfm}.

\begin{proposition}
	Let $\lambda$ and $\mu$ be cardinals, with $1 \leq \lambda \leq \mu$ and $\mu$ uncountable.
	\begin{enumerate}
		\item{$\bb{S}(\mu, \lambda)$ (resp. $\bb{S}(\mu, < \lambda)$) is $\omega_1$-closed and $(\mu + 1)$-strategically closed. In particular, forcing with $\bb{S}(\mu, \lambda)$ (resp. $\bb{S}(\mu, < \lambda)$) does not add any new $\mu$-sequences of ordinals.}
		\item{If $G$ is $\bb{S}(\mu, \lambda)$-generic (resp. $\bb{S}(\mu, < \lambda)$-generic) over $V$, then 
      $\bigcup G$ is a $\square_{\mu, \lambda}$-sequence (resp. a $\square_{\mu, < \lambda}$-sequence).}
	\end{enumerate}
\end{proposition}

There is also a natural forcing to add a thread through a square sequence.

\begin{definition}
	Suppose $\lambda$ and $\mu$ are cardinals and $\vec{\mathcal{C}} = \langle \mathcal{C}_\alpha \mid \alpha < \mu^+ \rangle$ is a $\square_{\mu, < \lambda}$-sequence. Let $\kappa \leq \mu$ be a regular, uncountable cardinal. $\bb{T}_\kappa(\vec{\mathcal{C}})$ is the forcing poset whose conditions are all $t$ such that:
	\begin{enumerate}
		\item{$t$ is a closed, bounded subset of $\mu^+$ and $\mathrm{otp}(t) < \kappa$;}
		\item{for all $\alpha \in t'$, $t \cap \alpha \in \mathcal{C}_\alpha$.}
	\end{enumerate}
	If $s,t \in \bb{T}_\kappa(\vec{\mathcal{C}})$, then $s \leq t$ iff $s$ end-extends $t$.
\end{definition}

If $\vec{\mathcal{C}}$ was added by forcing with $\bb{S}(\mu, < \lambda)$, then 
$\bb{T}_\kappa(\vec{\mathcal{C}})$ is rather nicely behaved. To be more precise, let 
$\dot{G}$ be the canonical $\bb{S}(\mu, < \lambda)$-name for the generic filter, and let $\dot{\vec{\mathcal{C}}}$ be the 
canonical $\bb{S}(\mu, < \lambda)$-name for $\bigcup \dot{G}$. Let $\kappa \leq \mu$ be a regular, uncountable cardinal. 
For a poset $\bb{P}$ and a cardinal $\theta$, let $\bb{P}^\theta$ denote the full-support product of $\theta$ copies of $\bb{P}$.

\begin{proposition} \label{denseclosed}
	Suppose $0 < \theta < \lambda \leq \mu^+$. Then $\bb{S}(\mu, < \lambda) * (\dot{\bb{T}}_\kappa(\dot{\vec{\mathcal{C}}}))^\theta$ 
  has a $\kappa$-directed closed dense subset.
\end{proposition} 

\begin{proof}
	Let $\bb{U}$ consist of all $(p, \langle \dot{t}_\eta \mid \eta < \theta \rangle) \in \bb{S}(\mu, < \lambda) * (\dot{\bb{T}}_\kappa(\dot{\vec{\mathcal{C}}}))^\theta$ such that:
	\begin{enumerate}
		\item{there is $\langle t_\eta \mid \eta < \theta \rangle \in V$ such that $p \Vdash ``\langle \dot{t}_\eta \mid \eta < \theta \rangle = \langle \check{t}_\eta \mid \eta < \theta \rangle"$;}
		\item{for all $\eta < \theta$, $\gamma^p = \max(t_\eta)$.}
	\end{enumerate}

	We first show that $\bb{U}$ is dense. To this end, let $(p_0, \langle \dot{t}^0_\eta \mid \eta < \theta \rangle) \in \bb{S}(\mu, < \lambda) * (\dot{\bb{T}}_\kappa(\dot{\vec{\mathcal{C}}}))^\theta$. Since $\bb{S}(\mu, < \lambda)$ is $\mu^+$-distributive, we can find $p \leq p_0$ and $\langle t^0_\eta \mid \eta < \theta \rangle \in V$ such that $p \Vdash ``\langle \dot{t}^0_\eta \mid \eta < \theta \rangle = \langle \check{t}^0_\eta \mid \eta < \theta \rangle."$ By strengthening $p$ if necessary, we may assume that $\gamma^p > \max(t^0_\eta)$ for all $\eta < \theta$. For all $\eta < \theta$, let $t_\eta = t^0_\eta \cup \{\gamma^p\}$. Then $(p, \langle \check{t}_\eta \mid \eta < \theta \rangle) \leq (p_0, \langle \dot{t}^0_\eta \mid \eta < \lambda \rangle)$ and is in $\bb{U}$.

  We next show that $\bb{U}$ is $\kappa$-directed closed. First note that $(\bb{U}, \geq)$ (with the reverse order) is tree-like, 
  i.e., for all $u,v,w \in \bb{U}$, if $w \leq u,v$, then either $u \leq v$ or $v \leq u$. A tree-like partial order is $\kappa$-directed 
  closed iff it is $\kappa$-closed, so it suffices to show that $\bb{U}$ is $\kappa$-closed. To this end, let $\langle (p_\xi, \langle \dot{t}^\xi_\eta \mid \eta < \theta \rangle) \mid \xi < \epsilon \rangle$ be a strictly decreasing sequence from $\bb{U}$, where $\epsilon < \kappa$ is a limit ordinal. For each $\xi < \epsilon$ and $\eta < \theta$, let $t^\xi_\eta$ be such that $p_\xi \Vdash ``\dot{t}^\xi_\eta = \check{t}^\xi_\eta."$ Let $\gamma = \sup(\{\gamma^{p_\xi} \mid \xi < \epsilon\})$. For each $\eta < \theta$, let $t^*_\eta = \bigcup_{\xi < \epsilon} t^\xi_\eta$, and note that $t^*_\eta$ is a club in $\gamma$. We define $p = \langle \mathcal{C}^p_\alpha \mid \alpha \leq \gamma \rangle$ as follows. For $\alpha < \gamma$, let $\mathcal{C}^p_\alpha = \mathcal{C}^{p_\xi}_\alpha$, where $\xi < \epsilon$ is least such that $\alpha \leq \gamma^{p_\xi}$. Let $\mathcal{C}^p_\gamma = \{t^*_\eta \mid \eta < \theta \}$. Finally, for all $\eta < \theta$, let $t_\eta = t^*_\eta \cup \{\gamma\}$. It is easily verified that $(p, \langle \check{t}_\eta \mid \eta < \theta \rangle)$ is a lower bound for $\langle (p_\xi, \langle \dot{t}^\xi_\eta \mid \eta < \theta \rangle) \mid \xi < \epsilon \rangle$ in $\bb{U}$.
\end{proof}

\begin{remark} \label{denseclosedremark}
  An easy adaptation to the proof of Proposition \ref{denseclosed} yields that if, in addition to the hypotheses 
  in Proposition \ref{denseclosed}, $\dot{\bb{R}}$ is a $\bb{S}(\mu, < \lambda) * \dot{\bb{T}}_\kappa(\dot{\vec{\mathcal{C}}})$-name 
  for a $\kappa$-closed forcing notion, then $\bb{S}(\mu, < \lambda) * (\dot{\bb{T}}_\kappa(\dot{\vec{\mathcal{C}}}) * \dot{\bb{R}})^\theta$ 
  has a dense $\kappa$-closed subset, namely the set $\bb{U}^*$ of all $(p, \langle (\dot{t}_\eta, \dot{r}_\eta) \mid \eta < \lambda \rangle)$ 
  such that $(p, \langle \dot{t}_\eta \mid \eta < \theta \rangle)$ is in the set $\bb{U}$ identified in the proof of Proposition \ref{denseclosed}. 
  The verification that $\bb{U}^*$ is dense and $\kappa$-closed is precisely as in the proof of Proposition \ref{denseclosed}.
\end{remark}

\begin{corollary} \label{distributive_cor}
	Let $G$ be $\bb{S}(\mu, < \lambda)$-generic over $V$, and let $\vec{\mathcal{C}} = \bigcup G$. Let $\kappa \leq \mu$ be a regular, uncountable cardinal.
	\begin{enumerate}
		\item{In $V[G]$, $\bb{T} = \bb{T}_\kappa(\vec{\mathcal{C}})$ is $\kappa$-distributive. Moreover, $\bb{T}^\theta$ is $\kappa$-distributive for all $\theta < \lambda$ 
      and, if $\dot{\bb{R}}$ is a $\bb{T}$-name for a $\kappa$-closed forcing, then $(\bb{T} * \dot{\bb{R}})^\theta$ is $\kappa$-distributive 
    for all $\theta < \lambda$.}
    \item{If $H$ is $\bb{T}$-generic over $V[G]$, then $T = \bigcup{H}$ is a thread through $\vec{\mathcal{C}}$ of order type $\kappa$.}
	\end{enumerate}
\end{corollary}

We similarly introduce a forcing poset to add a $\square(\kappa, \lambda)$-sequence.

\begin{definition}
	Suppose $1 < \lambda \leq \kappa$ are cardinals, with $\kappa$ regular and uncountable. $\bb{Q}(\kappa, < \lambda)$ is the forcing poset consisting of conditions $q = \langle \mathcal{C}^q_\alpha \mid \alpha \leq \gamma^q \rangle$ such that:
	\begin{enumerate}
		\item{$\gamma^q < \kappa$;}
		\item{for all limit $\alpha \leq \gamma^q$ and all $C \in \mathcal{C}^q_\alpha$, $C$ is a club in $\alpha$;}
		\item{for all limit $\alpha \leq \gamma^q$, $1 \leq |\mathcal{C}^q_\alpha| < \lambda$;}
		\item{for all limit $\alpha < \beta \leq \gamma^q$ and all $C \in \mathcal{C}^q_\beta$, if $\alpha \in C'$, then $C \cap \alpha \in \mathcal{C}^q_\alpha$.}
	\end{enumerate}
	$\bb{Q}(\kappa, < \lambda)$ is ordered by end-extension. $\bb{Q}(\kappa, < \lambda^+)$ will be denoted by $\bb{Q}(\kappa, \lambda)$.
\end{definition}

\begin{definition}
	Suppose $1 < \lambda \leq \kappa$ are cardinals, with $\kappa$ regular and uncountable, and $\vec{\mathcal{C}} = \langle \mathcal{C}_\alpha \mid \alpha < \kappa \rangle$ is a $\square(\kappa, < \lambda)$-sequence. $\bb{T}(\vec{\mathcal{C}})$ is the forcing poset whose conditions are closed, bounded subsets $t$ of $\kappa$ such that, for all $\alpha \in t'$, $t \cap \alpha \in \mathcal{C}_\alpha$. $\bb{T}(\vec{\mathcal{C}})$ is ordered by end-extension.
\end{definition}

The following is proved similarly to the corresponding facts about forcings to add and thread $\square_{\mu, < \lambda}$-sequences. A proof 
of item (2) for $\bb{Q}(\kappa, 1)$ can be found in \cite[Lemma 35]{lh}. The proof for $\bb{Q}(\kappa, < \lambda)$ for arbitrary $\lambda \leq \kappa$ 
is essentially the same.

\begin{proposition} \label{dist_prop_2}
	Let $1 < \lambda \leq \kappa$, with $\kappa$ regular and uncountable.
	\begin{enumerate}
		\item{$\bb{Q}(\kappa, < \lambda)$ is $\omega_1$-closed and $\kappa$-strategically closed.}
		\item{If $G$ is $\bb{Q}(\kappa, < \lambda)$-generic over $V$, then $\bigcup G$ is a $\square(\kappa, < \lambda)$-sequence in $V[G]$.}
		\item{If $\dot{\vec{\mathcal{C}}}$ is the canonical $\bb{Q}(\kappa, < \lambda)$-name for the union of the generic filter and $0 < \theta < \lambda$, then $\bb{Q}(\kappa, < \lambda) * \dot{\bb{T}}(\dot{\vec{\mathcal{C}}})^\theta$ has a dense $\kappa$-directed closed subset.}
	\end{enumerate}
\end{proposition}

We need two more general facts about forcing. The first is due to Magidor and concerns absorbing forcing posets into L\'{e}vy collapses.

\begin{fact}[\cite{magidor}, Lemma 3] \label{lift}
  Let $\kappa$ be a regular cardinal, and let $\kappa<\lambda<\mu$. Suppose that, in $V^{\mathrm{Coll}(\kappa, <\lambda)}$, $\mathbb{P}$ is a separative, $\kappa$-closed partial order and $|\mathbb{P}|<\mu$. Let $i$ be the natural complete embedding of $\mathrm{Coll}(\kappa, <\lambda)$ into $\mathrm{Coll}(\kappa, <\mu)$ (namely, the identity embedding). Then $i$ can be extended to a complete embedding $k$ of $\mathrm{Coll}(\kappa, <\lambda)*\dot{\mathbb{P}}$ into $\mathrm{Coll}(\kappa, <\mu)$ so that the quotient forcing $\mathrm{Coll}(\kappa, <\mu)/k``(\mathrm{Coll}(\kappa, <\lambda)*\dot{\mathbb{P}})$ is $\kappa$-closed.
\end{fact}

The second fact is due to Shelah and involves the preservation of stationary sets by sufficiently closed forcing.

\begin{fact}[\cite{shelah_col}, Theorem 20] \label{apstat}
	Let $\kappa < \lambda$ be infinite, regular cardinals. Suppose that $S$ is a stationary subset of $S^{\lambda}_{\kappa}$, $S \in I[\lambda]$, and $\mathbb{P}$ is a $\kappa^+$-closed forcing poset. Then $S$ remains stationary in $V^{\mathbb{P}}$.
\end{fact}

\section{Preservation lemmas} \label{preservationSect}

In this section, we present two preservation lemmas for narrow systems, each a slight improvement of a similar lemma of Sinapova \cite[Theorem 14]{sinapova}. 
For convenience, we first introduce the following definition.

\begin{definition} \label{full_set_def}
	Let $\lambda$ be an uncountable regular cardinal, let $S = \langle I \times \kappa, \mathcal{R} \rangle$ be a narrow $\lambda$-system, and let $\theta = \mathrm{width}(S)$. $\bar{b} = \{b_\gamma \mid \gamma < \theta \}$ is a \emph{full set of branches through S} if:
	\begin{enumerate}
    \item{for all $\gamma < \theta$, there is $R \in \mathcal{R}$ such that $b_\gamma$ is a branch of $S$ through $R$;}
		\item{for all $\alpha \in I$, there is $\gamma < \theta$ such that $b_\gamma \cap S_\alpha \not= \emptyset$.}
	\end{enumerate}
\end{definition}

\begin{remark}
	Note that, since $\lambda$ is regular and $\theta < \lambda$, condition (2) in the above definition implies that, for some $\gamma < \theta$, $b_\gamma$ is a cofinal branch.
\end{remark}

\begin{lemma} \label{preservationtheorem}
	Suppose that $\lambda$ is an uncountable cardinal, $S = \langle I \times \kappa, \mathcal{R} \rangle$ is a narrow $\lambda$-system, $\mathrm{width}(S) = \theta$, $\bb{P}$ is a $\theta^+$-closed forcing poset, $G$ is $\bb{P}$-generic over $V$, and, in $V[G]$, there is a full set of branches through $S$. Then there is a cofinal branch through $S$ in $V$.
\end{lemma}

\begin{remark}
  This lemma improves Sinapova's lemma in that it applies to $\theta^+$-closed forcing posets, whereas \cite[Theorem 14]{sinapova} applies only to $\theta^{++}$-closed posets.
\end{remark}

\begin{proof}
	We work in $V$, supposing for the sake of contradiction that there is no cofinal branch through $S$. For $\gamma < \theta$, let $\dot{b}_\gamma$ be a $\bb{P}$-name, and let $p^* \in \bb{P}$ be such that $p^* \Vdash ``\{\dot{b}_\gamma \mid \gamma < \theta\}$ is a full set of branches through S.$"$ Since $\bb{P}$ is $\theta^+$-closed, we may assume that there are a nonempty $A \subseteq \theta$, $\alpha^* < \lambda$, and $r:A \rightarrow \mathcal{R}$ such that:
	\begin{itemize}
		\item{for all $\gamma < \theta, p^* \Vdash ``\dot{b}_\gamma$ is a cofinal branch iff $\gamma \in A"$;}
		\item{for all $\gamma \in \theta \setminus A$, $p^* \Vdash ``\dot{b}_\gamma \subseteq S_{<\alpha^*}"$;}
		\item{for all $\gamma \in A$, $p^* \Vdash ``\dot{b}_\gamma$ is a branch through $r(\gamma)."$}
	\end{itemize}
	\begin{claim} \label{split_claim}
		For every $p \leq p^*$ and every $\gamma \in A$, there are $q_0, q_1 \leq p$ and $u_0, u_1 \in S$ such that:
		\begin{enumerate}
			\item{for $i < 2$, $q_i \Vdash ``u_i \in \dot{b}_\gamma"$;}
			\item{$u_0 \perp_{r(\gamma)} u_1$.}
		\end{enumerate}
	\end{claim}
	\begin{proof}
		Suppose not, and let $p$ and $\gamma$ form a counterexample. Then $b = \{u \in S \mid$ for some $q \leq p$, $q \Vdash ``u \in \dot{b}_\gamma"\}$ is a cofinal branch through $S$ in $V$.
	\end{proof}
	\begin{claim} \label{oneextensionclaim}
		For every $p_0, p_1 \leq p^*$ and every $\gamma \in A$, there are $q_0 \leq p_0$, $q_1 \leq p_1$, and $u_0, u_1 \in S$ such that:
		\begin{enumerate}
			\item{for $i < 2$, $q_i \Vdash ``u_i \in \dot{b}_\gamma"$;}
			\item{$u_0 \perp_{r(\gamma)} u_1$.}
		\end{enumerate}
	\end{claim}
	\begin{proof}
		First, apply Claim \ref{split_claim} to obtain $q^0_0, q^0_1 \leq p_0$ and $u^0_0, u^0_1 \in S$ such that $q^0_i \Vdash ``u^0_i \in \dot{b}_\gamma"$ and $u^0_0 \perp_{r(\gamma)} u^0_1$. Let $\beta^* <\lambda$ be such that $u^0_0, u^0_1 \in S_{<\beta^*}$. Find $q_1 \leq p_1$ and $u_1 \in S_{\geq \beta^*}$ such that $q_1 \Vdash ``u_1 \in \dot{b}_\gamma."$ If $(u^0_0, u_1)$ and $(u^0_1, u_1)$ are both in $r(\gamma)$, then, since $r(\gamma)$ is tree-like, $u^0_0$ and $u^0_1$ are $r(\gamma)$-comparable, which is a contradiction. Thus, there is $i^* < 2$ such that $u^0_{i^*} \perp_{r(\gamma)} u_1$. Let $q_0 = q^0_{i^*}$ and $u_0 = u^0_{i^*}$. Then $q_0, q_1, u_0$, and $u_1$ are as desired.
	\end{proof}
	\begin{claim} \label{manyextensionclaim}
		For every $p \leq p^*$, there are $q_0, q_1 \leq p$ and $\{u^\gamma_i \mid \gamma \in A, i < 2\} \subseteq S$ such that:
		\begin{enumerate}
			\item{for every $\gamma \in A$ and $i < 2$, $q_i \Vdash ``u^\gamma_i \in \dot{b}_\gamma"$;}
			\item{for every $\gamma \in A$, $u^\gamma_0 \perp_{r(\gamma)} u^\gamma_1$.}
		\end{enumerate}
	\end{claim}
	\begin{proof}
		We recursively build two decreasing sequences, $\langle q^0_\gamma \mid \gamma < \theta \rangle$ and $\langle q^1_\gamma \mid \gamma < \theta \rangle$ from $\bb{P}$ together with nodes from $S$, $\{u^\gamma_i \mid \gamma \in A, i < 2\}$, as follows.

		First, let $q^0_0 = q^1_0 = p$. If $\gamma < \theta$ is a limit ordinal and $i < 2$, let $q^i_\gamma$ be a lower bound for $\langle q^i_\xi \mid \xi < \gamma \rangle$. If $q^0_\gamma, q^1_\gamma$ have been defined and $\gamma \not\in A$, let $q^i_{\gamma + 1} = q^i_\gamma$ for $i < 2$. Finally, if $q^0_\gamma$ and $q^1_\gamma$ have been defined and $\gamma \in A$, apply Claim \ref{oneextensionclaim} to $q^0_\gamma$, $q^1_\gamma$, and $\gamma$ to obtain $q^0_{\gamma+1} \leq q^0_\gamma$, $q^1_{\gamma+1} \leq q^1_\gamma$, and $u^\gamma_0$, $u^\gamma_1 \in S$ such that:
		\begin{itemize}
			\item{for $i < 2$, $q^i_{\gamma + 1} \Vdash ``u^\gamma_i \in \dot{b}_\gamma"$;}
			\item{$u^\gamma_0 \perp_{r(\gamma)} u^\gamma_1$.}
		\end{itemize}
		At the end of the construction, for $i < 2$, let $q_i$ be a lower bound for $\langle q^i_\gamma \mid \gamma < \theta \rangle$. Then $q_0, q_1$, and $\{u^\gamma_i \mid \gamma \in A, i < 2\}$ are as desired.
	\end{proof}
	Now use Claim \ref{manyextensionclaim} and the closure of $\bb{P}$ to recursively build a tree of conditions $\{p_\sigma \mid \sigma \in {^{<\theta}}2 \}$ and nodes $\{u^{\sigma, \gamma}_i \mid \sigma \in {^{<\theta}}2, \gamma \in A, i < 2 \}$ in $S$ as follows.

	Let $p_\emptyset = p^*$. If $\eta < \theta$ is a limit ordinal, $\sigma \in {^\eta}2$, and $p_{\sigma \restriction \xi}$ has been defined for all $\xi < \eta$, let $p_\sigma$ be a lower bound for $\langle p_{\sigma \restriction \xi} \mid \xi < \eta \rangle$. If $\sigma \in {^{<\theta}}2$ and $p_\sigma$ has been defined, apply Claim \ref{manyextensionclaim} to $p_\sigma$ to obtain $p_{\sigma ^\frown \langle 0 \rangle}$, $p_{\sigma ^\frown \langle 1 \rangle} \leq p_\sigma$ and nodes $\{u^{\sigma, \gamma}_i \mid \gamma \in A, i < 2\}$ in $S$ such that:
	\begin{itemize}
		\item{for every $\gamma \in A$ and $i < 2$, $p_{\sigma ^\frown \langle i \rangle} \Vdash ``u^{\sigma, \gamma}_i \in \dot{b}_\gamma;"$}
		\item{for every $\gamma \in A$, $u^{\sigma, \gamma}_0 \perp_{r(\gamma)} u^{\sigma, \gamma}_1$.}
	\end{itemize} 
	For each $f \in {^\theta}2$, let $p_f$ be a lower bound for $\langle p_{f \restriction \eta} \mid \eta < \theta \rangle$. Choose $B \subseteq {^\theta}2$ with $|B| = \theta^+$, and find $\beta^* \in (I \setminus \alpha^*)$ such that $u^{f \restriction \eta, \gamma}_i \in S_{<\beta^*}$ for all $f \in B$, $\eta < \theta$, $\gamma \in A$, and $i < 2$. This is possible, since $\theta^+ < \lambda$ by the assumption that $S$ is a narrow system.

	For each $f \in B$, find $q_f \leq p_f$, $v_f \in S_{\beta^*}$, and $\gamma_f \in A$ such that $q_f \Vdash ``v_f \in \dot{b}_{\gamma_f}."$ Since $|B| = \theta^+$, we can find $v \in S_{\beta^*}$, $\gamma \in A$, and $f \not= g$, both in $B$, such that $v_f = v_g = v$ and $\gamma_f = \gamma_g = \gamma$. Let $\eta^*$ be the least $\eta < \theta$ such that $f(\eta) \not= g(\eta)$, and let $\sigma = f \restriction \eta^* = g \restriction \eta^*$. Assume without loss of generality that $f(\eta^*) = 0$ and $g(\eta^*) = 1$. Then $q_f \Vdash ``u^{\sigma, \gamma}_0, v \in \dot{b}_\gamma"$, so $u^{\sigma, \gamma}_0 <_{r(\gamma)} v$. Similarly, $q_g \Vdash ``u^{\sigma, \gamma}_1, v \in \dot{b}_\gamma"$, so $u^{\sigma, \gamma}_1 <_{r(\gamma)} v$. Thus, since $r(\gamma)$ is tree-like, $u^{\sigma, \gamma}_0$ and $u^{\sigma, \gamma}_1$ are $r(\gamma)$-comparable, contradicting $u^{\sigma, \gamma}_0 \perp_{r(\gamma)} u^{\sigma, \gamma}_1$.
\end{proof}

The next variation on Sinapova's theorem is due to Neeman.

\begin{lemma}[\cite{neeman}, Lemma 3.3 and Remark 3.4] \label{neemanlemma}
	Suppose that $\lambda$ is a regular, uncountable cardinal, $S = \langle I \times \kappa, \mathcal{R} \rangle$ is a narrow $\lambda$-system, and $\mathrm{width}(S) = \theta$. Suppose $\bb{P}$ is a forcing poset, and let $\bb{P}^{\theta^+}$ denote the full-support product of $\theta^+$ copies of $\bb{P}$. Suppose moreover that $\bb{P}^{\theta^+}$ is $\theta^{++}$-distributive, $G$ is $\bb{P}$-generic over $V$, and, in $V[G]$, there is a full set of branches through $S$. Then there is a cofinal branch through $S$ in $V$.
\end{lemma}

\section{Weak square and the narrow system property} \label{weakSquareSect}

In this section, we obtain some consistency results about the narrow system property. In particular, we show that, 
unlike the tree property, the narrow system property is compatible with certain weak square principles. 
We also prove a global result about the consistency of the statement that all narrow systems have a cofinal branch. 
We first give a useful result about forcing the narrow system property at small cardinals.

\begin{theorem} \label{indestthm}
	Let $\mu < \lambda$ be infinite cardinals, with $\mu$ regular and $\lambda$ supercompact. Let $\bb{P} = \mathrm{Coll}(\mu, < \lambda)$. Then, in $V^{\mathbb{P}}$, $NSP(\mu, \geq \lambda)$ holds and moreover is indestructible under $\lambda$-directed closed set forcing.
\end{theorem}

\begin{proof}
	Let $G$ be $\bb{P}$-generic over $V$. Since trivial forcing is $\lambda$-directed closed, it suffices to prove that if, in $V[G]$, $\bb{Q}$ is a $\lambda$-directed closed forcing poset and $H$ is $\bb{Q}$-generic over $V[G]$, then $NSP(\mu, \geq \lambda)$ holds in $V[G*H]$.

	Thus, let $\bb{Q}$ be $\lambda$-directed closed in $V[G]$, and let $H$ be $\bb{Q}$-generic over $V[G]$. 
  In $V[G*H]$, let $\nu \geq \lambda$ be a regular cardinal, 
  let $\kappa < \mu$, and let $S = \langle I \times \kappa, \mathcal{R} \rangle$ be a narrow $\nu$-system. 
  As usual, we assume that $\mathcal{R} = \{R\}$.

  In $V$, let $\dot{\bb{Q}}$ be a $\bb{P}$-name for $\bb{Q}$ and fix a cardinal $\delta > \nu$ such that $\bb{P} * \dot{\bb{Q}} \in V_\delta$. 
  Let $j:V \rightarrow M$ witness that $\lambda$ is $\delta$-supercompact. In particular, $\mathrm{crit}(j) = \lambda$, 
  $j(\lambda) > \delta$, and ${^\delta}M \subseteq M$. $j(\bb{P}) = \mathrm{Coll}(\mu, < j(\lambda))$, so, 
  by Fact \ref{lift}, the identity embedding $i:\bb{P} \rightarrow j(\bb{P})$ can be extended to a complete embedding 
  $k:\bb{P} * \dot{\bb{Q}} \rightarrow j(\bb{P})$ so that the quotient forcing $j(\bb{P})/j``(\bb{P} * \dot{\bb{Q}})$ is 
  $\mu$-closed. Therefore, letting $\dot{\bb{R}}$ be a $\bb{P} * \dot{\bb{Q}}$-name for this quotient forcing, 
  we have $j(\bb{P}) \cong \bb{P} * \dot{\bb{Q}} * \dot{\bb{R}}$ and $\bb{R}$ is $\mu$-closed in 
  $V^{\bb{P} * \dot{\bb{Q}}}$. Let $K$ be $\bb{R}$-generic over $V[G*H]$. Since $j(p) = p$ for all $p \in \bb{P}$, 
  we have $j``G \subseteq G*H*K$, so, in $V[G*H*K]$, we can extend $j$ to $j:V[G] \rightarrow M[G*H*K]$.

  If $\dot{q} \in V$ is a $\bb{P}$-name for an element of $\dot{\bb{Q}}$, let $r_{\dot{q}}$ be the interpretation of $j(\dot{q})$ in $V[G*H*K]$. 
  Let $\bar{H} = \{r_{\dot{q}} \mid$ for some $p \in \bb{P}$, $(p, \dot{q}) \in G*H\}$. Since $j \restriction \bb{P} * \dot{\bb{Q}} \in M$, 
  $\bar{H} \in M[G*H*K]$. Moreover, $\bar{H}$ is a directed subset of $j(\bb{Q})$,  
  $|\bar{H}| < \delta < j(\lambda)$, and $j(\bb{Q})$ is $j(\lambda)$-directed closed in $M[G*H*K]$, so there is 
  $q^* \in j(\bb{Q})$ such that $q^* \leq r$ for all $r \in \bar{H}$. 
  Let $H^+$ be $j(\bb{Q})$-generic over $V[G*H*K]$ with $q^* \in H^+$. By construction, we have $j``H \subseteq H^+$, so we may, in 
  $V[G*H*K*H^+]$, further extend $j$ to $j:V[G*H] \rightarrow M[G*H*K*H^+]$.

	In $M[G*H*K*H^+]$, $j(S) = \langle j(I) \times \kappa, \{j(R)\} \rangle$ is a $j(\nu)$-system. 
  Let $\eta = \sup(j``\nu) < j(\nu)$, and let $\xi = \min(j(I) \setminus \eta)$. For each $\gamma < \kappa$, 
  let $I_\gamma = \{\alpha \in I \mid$ for some $\beta < \kappa, (j(\alpha), \beta) <_{j(R)} (\xi, \gamma)\}$. 
  Fix $\gamma < \kappa$ and $\alpha \in I_\gamma$. Since $j(R)$ is tree-like and any two distinct elements of $j(S)_{j(\alpha)}$ 
  are $j(R)$-incomparable with one another, there is a unique $\beta < \kappa$ such that $(j(\alpha), \beta) <_{j(R)} (\xi, \gamma)$. 
  Denote this unique $\beta < \kappa$ by $\beta^\alpha_\gamma$. For $\gamma < \kappa$, let $b_\gamma = 
  \{(\alpha, \beta^\alpha_\gamma) \mid \alpha \in I_\gamma \}$.
	
  $\bar{b} := \{b_\gamma \mid \gamma < \kappa \} \in V[G*H*K*H^+]$, and we claim that $\bar{b}$ is a full set of 
  branches through $S$. We first verify clause (1) of Definition \ref{full_set_def}. To this end, fix $\gamma < \kappa$ 
  and $\alpha, \alpha' \in I_\gamma$. In $M[G*H*K*H^+]$, $(j(\alpha), \beta^\alpha_\gamma), (j(\alpha'), \beta^{\alpha'}_\gamma) 
  <_{j(R)} (\xi, \gamma)$, so, since $j(R)$ is tree-like, $(j(\alpha), \beta^\alpha_\gamma)$ and $(j(\alpha'), \beta^{\alpha'}_\gamma)$ 
  are $j(R)$-comparable. By elementarity and the fact that $j$ is the identity on $\kappa$, it is the case that $(\alpha, \beta^\alpha_\gamma)$ 
  and $(\alpha', \beta^{\alpha'}_\gamma)$ are $R$-comparable. Therefore, $b_\gamma$ is a branch of $S$. To verify clause (2) of Definition \ref{full_set_def}, 
  fix $\alpha \in I$. By the definition of system, 
  there are $\beta, \gamma < \kappa$ such that $(j(\alpha), \beta) <_{j(R)} (\xi, \gamma)$. For these values of $\beta$ and $\gamma$, 
  we then have $\alpha \in I_\gamma$ and $\beta = \beta^\alpha_\gamma$, so $(\alpha, \beta) \in b_\gamma \cap S_\alpha$.
  
  Since $\bb{P} * \dot{\bb{Q}} * \dot{\bb{R}}$ is $\mu$-closed in $V$, $\mu < \delta$, and $M$ is closed under $\delta$-sequences 
  in $V$, $M[G*H*K]$ is closed under $<\mu$-sequences in $V[G*H*K]$. Therefore, since $j(\bb{Q})$ is $j(\lambda)$-directed closed 
  in $M[G*H*K]$ and $j(\lambda) > \delta$, it is $\mu$-directed closed in $V[G*H*K]$. Thus, $\bb{R} * j(\dot{\bb{Q}})$ is 
  $\mu$-closed in $V[G*H]$, so, since $\mathrm{width}(S) = \kappa < \mu$, Lemma \ref{preservationtheorem} implies that $S$ has a cofinal branch in $V[G*H]$. 
\end{proof}

We can use this to get a global result.

\begin{theorem}
	Suppose there is a proper class of supercompact cardinals. Then there is a class forcing extension in which every narrow system has a cofinal branch.
\end{theorem}

\begin{proof}
	Let $\langle \kappa_i \mid i \in \mathrm{On} \rangle$ be an increasing, continuous sequence of cardinals such that:
	\begin{itemize}
		\item{$\kappa_0 = \omega$;}
		\item{if $i = 0$ or $i$ is a successor ordinal, then $\kappa_{i+1}$ is supercompact;}
		\item{if $i$ is a limit ordinal, then $\kappa_{i+1} = \kappa_i^+$.}
	\end{itemize}
	We may assume that, if $i$ is a limit ordinal, then $\kappa_i$ is singular, for, if this is not the case, then we may let $i$ be least such that $\kappa_i$ is regular and work in $V_{\kappa_i}$ instead of $V$.

  Informally, we force with a class-length iteration of L\'{e}vy collapses to turn each $\kappa_i$ into $\aleph_i$. More precisely, we recursively define posets $\langle \bb{P}_i \mid i \in \mathrm{On} \rangle$ as follows.
	\begin{itemize}
		\item{$\bb{P}_0$ is trivial forcing.}
    \item{If $i = 0$ or $i$ is a successor ordinal, then $\bb{P}_{i+1} = \bb{P}_i * \dot{\mathrm{Coll}}(\kappa_i, < \kappa_{i+1})$.}
		\item{If $i$ is a limit ordinal, then $\bb{P}_i$ is the inverse (i.e., full-support) limit of 
      $\langle \bb{P}_j \mid j < i \rangle$ and $\bb{P}_{i+1} = \bb{P}_i * \{\dot{\bbm{1}}\}$, where $\dot{\{\bbm{1}\}}$ is a $\bb{P}_i$-name for trivial forcing.}
	\end{itemize}
	For ordinals $i < k$, let $\dot{\bb{P}}_{i,k}$ be a $\bb{P}_i$-name such that $\bb{P}_k \cong \bb{P}_i * \dot{\bb{P}}_{i,k}$ and note that, in $V^{\bb{P}_i}$, $\bb{P}_{i,k}$ is $\kappa_i$-directed closed. Thus, $(H(\kappa_i))^{V^{\bb{P}_i}} = (H(\kappa_i))^{V^{\bb{P}_k}}$, so $V^{\bb{P}} = \bigcup_{i \in \mathrm{On}} V^{\bb{P}_i}$ is a model of ZFC. Also, standard arguments show that, in $V^\bb{P}$, for all $i \in \mathrm{On}$, $\kappa_i = \aleph_i$, i.e., $\{\kappa_i \mid i \in \mathrm{On}\}$ are precisely the infinite cardinals of $V^{\bb{P}}$. Moreover, for all $i \in \mathrm{On}$, $\kappa_i$ is regular in $V^{\bb{P}}$ iff $i$ is $0$ or a successor ordinal.

  We now show that, in $V^{\bb{P}}$, every narrow system has a cofinal branch. Since the infinite cardinals of $V^{\bb{P}}$ 
  are precisely $\{\kappa_i \mid i \in \mathrm{On}\}$ and, for an ordinal $i$, $\kappa_i$ is regular iff $i=0$ or $i$ is a successor ordinal, 
  it suffices to show that, for all ordinals $i, k$ with $i+1 < k$ and $k$ a successor ordinal, every system with width 
  $\kappa_i$ and height $\kappa_k$ has a cofinal branch. Since, in $V^{\bb{P}_{k+1}}$, $\bb{P}_{k+1, \ell}$ is $\kappa_{k+1}$-closed 
  for all $\ell \geq k+1$, all such systems are in $V^{\bb{P}_{k+1}}$. We may as usual just deal with systems with a single relation.

  In $V^{\bb{P}_{i+1}}$, $\kappa_{i+1} = \kappa_i^+$ and, since $|\bb{P}_{i+1}| < \kappa_{i+2}$, a result of L\'{e}vy and Solovay 
  (see \cite{levysolovay}) yields that $\kappa_{i+2}$ is still supercompact. $\bb{P}_{i+1, i+2} = \mathrm{Coll}(\kappa_{i+1}, < \kappa_{i+2})$, 
  so, by Theorem \ref{indestthm}, $NSP(\kappa_{i+1}, \geq \kappa_{i+2})$ holds in $V^{\bb{P}_{i+2}}$ and is indestructible under $\kappa_{i+2}$-directed 
  closed set forcing. Since $\bb{P}_{i+2, k+1}$ is $\kappa_{i+2}$-directed closed, $NSP(\kappa_{i+1}, \geq \kappa_{i+2})$ holds in $V^{\bb{P}_{k+1}}$. 
  Since $\kappa_k \geq \kappa_{i+2}$, we have that, in $V^{\bb{P}_{k+1}}$, every system with width $\kappa_i$ and height $\kappa_k$ has a cofinal branch. 
\end{proof}

\begin{theorem} \label{squarenspthm}
  Suppose there are infinitely many supercompact cardinals. Then there is a forcing extension 
  in which $\square_{\aleph_\omega, < \aleph_\omega}$ and $NSP(\aleph_{\omega+1})$ both hold.
\end{theorem}

\begin{proof}
  Let $\langle \kappa_n \mid n < \omega \rangle$ be an increasing sequence of supercompact cardinals, let $\mu = \sup(\{\kappa_n \mid n < \omega\})$, and let $\lambda = \mu^+$. Define an iteration $\langle \bb{P}_n, \bb{\dot{Q}}_n \mid n < \omega \rangle$ by letting $\bb{P}_0$ be trivial forcing, 
  $\bb{Q}_0 = \mathrm{Coll}(\omega, < \kappa_0)$, and, for all $n < \omega$, letting $\bb{\dot{Q}}_{n+1}$ be a $\bb{P}_{n+1}$-name for $\mathrm{Coll}(\kappa_n, < \kappa_{n+1})$. Let $\mathbb{P}$ be the inverse limit of the iteration. Thus, in $V^{\bb{P}}$, $\kappa_n = \aleph_{n+1}$ for all $n < \omega$, $\mu = \aleph_\omega$, and $\lambda = \aleph_{\omega + 1}$. For all $n < \omega$, let $\bb{\dot{P}}^n$ be a $\bb{P}_n$-name such that $\bb{P} \cong \bb{P}_n * \bb{\dot{P}}^n$ and, for all $m < n < \omega$, let $\dot{\bb{P}}_{mn}$ be a $\bb{P}_m$-name such that $\bb{P}_n \cong \bb{P}_m * \dot{\bb{P}}_{mn}$.

  In $V^{\bb{P}}$, let $\bb{S} = \bb{S}(\mu, < \mu)$, the forcing introduced in Section \ref{forcingSect} 
  to add a $\square_{\mu, < \mu}$-sequence. We claim that $V^{\bb{P}*\dot{\bb{S}}}$ is the desired model. 
  To this end, let $G$ be $\bb{P}$-generic over $V$, and let $H$ be $\bb{S}$-generic over $V[G]$. For 
  $n < \omega$, let $G_n$ and $G^n$ be the generic filters induced by $G$ on $\bb{P}_n$ and $\bb{P}^n$, 
  respectively. For $m < n < \omega$, let $G_{mn}$ be the generic filter induced by $G$ on $\bb{P}_{mn}$. 
  Suppose for sake of contradiction that, in $V[G*H]$, there is a narrow $\lambda$-system with no cofinal branch. 
  As usual, we may assume that there is such a system of the form 
  $S = \langle \lambda \times \kappa, \{R\} \rangle$ for some $\kappa < \mu$.

  Let $\vec{\mathcal{C}} = \bigcup H$. $\vec{\mathcal{C}} = \langle \mathcal{C}_\alpha \mid \alpha < \lambda \rangle$ is thus a 
  $\square_{\mu, <\mu}$-sequence in $V[G*H]$. Let $n^* < \omega$ be least such that $\kappa \leq \kappa_{n^*}$, let 
  $\kappa^* = \kappa_{n^* + 3}$, and let $\bb{T} = \bb{T}_{\kappa^*}(\vec{\mathcal{C}})$. Since $|\bb{P}_{n^*+3}| < \kappa^*$, 
  the aforementioned result of L\'{e}vy and Solovay from \cite{levysolovay} implies that
  $\kappa^*$ remains supercompact in $V[G_{n^*+3}]$. Fix a cardinal $\delta$ such that $\bb{P}^{n^*+3} * \dot{\bb{S}} * \dot{\bb{T}} 
  \in (V_\delta)^{V[G_{n^*+3}]}$, and let $j:V[G_{n^*+3}] \rightarrow M[G_{n^*+3}]$ witness that $\kappa^*$ is 
  $\delta$-supercompact, i.e., $\mathrm{crit}(j) = \kappa^*$, $j(\kappa^*) > \delta$, and ${^\delta}M[G_{n^*+3}] 
  \subseteq M[G_{n^*+3}]$. Note that $j(\bb{Q}_{n^*+3}) = j(\mathrm{Coll}(\kappa_{n^*+2}, < \kappa^*)) = 
  \mathrm{Coll}(\kappa_{n^*+2}, < j(\kappa^*))$. In $V[G_{n^*+4}]$, by Proposition \ref{denseclosed}, $\bb{P}^{n^*+4} * \dot{\bb{S}} * \dot{\bb{T}}$ 
  has a dense $\kappa^*$-directed closed subset and is of size less than $j(\kappa^*)$, so, by Fact \ref{lift}, in 
  $V[G_{n^*+3}]$, the identity embedding from $\bb{Q}_{n^*+3}$ into $j(\bb{Q}_{n^*+3})$ can be extended to a complete 
  embedding of $\bb{P}^{n^*+3} * \dot{\bb{S}} * \dot{\bb{T}}$ into $j(\bb{Q}_{n^*+3})$ in such a way that the quotient 
  forcing is $\kappa_{n^*+2}$-closed. Let $\dot{\bb{R}}$ be a $\bb{P}^{n^*+3} * \dot{\bb{S}} * \dot{\bb{T}}$-name 
  for this quotient forcing. Then $j(\bb{Q}_{n^*+3}) \cong \bb{P}^{n^*+3} * \dot{\bb{S}} * \dot{\bb{T}} * \dot{\bb{R}}$. Let $I$ be 
  $\bb{T}$-generic over $V[G*H]$ and $J$ be $\bb{R}$-generic over $V[G*H*I]$. Since $j$ is the identity on $\bb{Q}_{n^*+3}$, 
  we have $j``G_{n^*+3, n^*+4} \subseteq G^{n^*+3} * H * I * J$, so we may extend $j$ in $V[G*H*I*J]$ to a map 
  $j:V[G_{n^*+4}] \rightarrow M[G*H*I*J]$.

  We would like to extend $j$ further to have domain $V[G * H]$. To do this, we will construct a master condition in 
  $j(\bb{P}^{n^* + 4} * \dot{\bb{S}})$. In $V[G_{n^*+4}]$, let $\bb{U}$ be the dense $\kappa^*$-directed 
  closed subset of $\bb{P}^{n^*+4} * \dot{\bb{S}} * \dot{\bb{T}}$, and let $\dot{\bb{U}} \in V[G_{n^*+3}]$ be a 
  $\bb{Q}_{n^*+3}$-name for $\bb{U}$. If $\dot{s} \in V[G_{n^*+3}]$ is a $\bb{Q}_{n^*+3}$-name 
  for an element of $\dot{\bb{P}}^{n^*+4} * \dot{\bb{S}} * \dot{\bb{T}}$, let $r_{\dot{s}}$ be the interpretation 
  of $j(\dot{s})$ in $V[G*H*I*J]$. Let $\bar{H} = \{r_{\dot{s}} \mid$ for some $q \in \bb{Q}_{n^*+3}$, 
  $(q, \dot{s}) \in G^{n^*+3} * H * I$ and $q \Vdash ``\dot{s} \in \dot{\bb{U}}"\}$. Since 
  $j \restriction \bb{P}^{n^*+3} * \dot{\bb{S}} * \dot{\bb{T}} \in M[G_{n^*+3}]$, 
  we have $\bar{H} \in M[G*H*I*J]$. Moreover, $\bar{H}$ is a directed subset of $j(\bb{U})$, 
  $|\bar{H}| < \delta < j(\kappa^*)$, and $j(\bb{U})$ is $j(\kappa^*)$-directed closed in $M[G*H*I*J]$. 
  Therefore, we can find $r^* \in j(\bb{U})$ such that $r^* \leq r$ for all $r \in \bar{H}$. 
  Let $r^* = (p^*, \dot{s}^*, \dot{t}^*)$, where $(p^*, \dot{s}^*) \in j(\bb{P}^{n^*+4} * \dot{\bb{S}})$. 
  Let $G^+ * H^+$ be $j(\bb{P}^{n^*+4} * \dot{\bb{S}})$-generic over $V[G*H*I*J]$ with $(p^*, \dot{s}^*) 
  \in G^+ * H^+$. By construction, $j``G^{n^*+4}*H \subseteq G^+ * H^+$, so, in $V[G*H*I*J*G^+*H^+]$, 
  we can extend $j$ to $j:V[G*H] \rightarrow M[G*H*I*J*G^+*H^+]$.
  
  In $M[G*H*I*J*G^+*H^+]$, $j(S) = \langle j(\lambda) \times \kappa, \{j(R)\} \rangle$ is a 
  $j(\lambda)$-system. Recall that $\eta = \sup(j``\lambda)$. For each $\gamma < \kappa$, let 
  $b_\gamma = \{u \in S \mid j(u) <_{j(R)} (\eta,\gamma)\}$. It is easily verified, exactly 
  as was done in the proof of Theorem \ref{indestthm}, that 
  $\{b_\gamma \mid \gamma < \kappa \}$ is a full set of branches through $S$.

  Since ${^\delta}M[G_{n^*+3}] \subseteq M[G_{n^*+3}]$ in $V[G_{n^*+3}]$, $\kappa_{n^*+2} < \delta$, and $\bb{P}^{n^*+3} * 
  \dot{\bb{S}} * \dot{\bb{T}}$ has a dense $\kappa_{n^*+2}$-directed closed subset in $V[G_{n^*+3}]$, we have 
  ${^{<\kappa_{n^*+2}}}M[G*H*I] \subseteq M[G*H*I]$ in $V[G*H*I]$. Therefore, since $\bb{R} * j(\bb{P}^{n^*+4} 
  * \dot{\bb{S}} * \dot{\bb{T}})$ has a dense $\kappa_{n^*+2}$-closed subset in $M[G*H*I]$, it has a 
  dense $\kappa_{n^*+2}$-closed subset in $V[G*H*I]$ as well. Combined with Clause (1) of Corollary \ref{distributive_cor}, this 
  implies that $(\bb{T} * \dot{\bb{R}} * j(\bb{P}^{n^*+4} * \dot{\bb{S}} * \dot{\bb{T}}))^{\kappa_{n^*+1}}$ is 
  $\kappa_{n^*+2}$-distributive in $V[G*H]$, so, \emph{a fortiori}, $(\bb{T} * \dot{\bb{R}} * j(\bb{P}^{n^*+4} * \dot{\bb{S}}))^{\kappa_{n^*+1}}$ 
  is $\kappa_{n^*+2}$-distributive in $V[G*H]$. Therefore, since $\{b_\gamma \mid \gamma < \kappa \} \in V[G*H*I*J*G^+*H^+]$ 
  is a full set of branches through $S$, Lemma \ref{neemanlemma} implies that there is a cofinal branch through $S$ in $V[G*H]$.
\end{proof}

\begin{remark}
  If a tree admits a narrow system that has a cofinal branch, then the downward closure of this cofinal branch in the tree is  
  a cofinal branch through the tree. Since, by the results mentioned in Remark \ref{square_remark}, $\square_{\aleph_\omega, < \aleph_\omega}$ implies the existence of a 
  special $\aleph_{\omega + 1}$-Aronszajn tree, the model from Theorem \ref{squarenspthm} provides an example 
  of a model in which there are $\aleph_{\omega+1}$-trees that do not admit narrow systems (in particular, 
  any $\aleph_{\omega+1}$-Aronszajn tree in that model cannot admit a narrow system).
\end{remark}

We can use a similar argument to show that, unlike the tree property, the narrow system property is not equivalent to weak compactness for inaccessible cardinals.

\begin{theorem}
	Suppose $\lambda$ is a supercompact cardinal. There is a forcing extension in which $\lambda$ remains inaccessible, $NSP(\lambda, \geq \lambda)$ holds, and $\lambda$ is not weakly compact.
\end{theorem}

\begin{proof}
	By forcing with the Laver preparation (see \cite{laver}), we may assume that the supercompactness of 
  $\lambda$ is indestructible under $\lambda$-directed closed forcing. Let $\bb{Q} = \bb{Q}(\lambda, < \lambda)$, 
  the standard forcing introduced in Section \ref{forcingSect} to add a $\square(\lambda, < \lambda)$-sequence. 
  Let $G$ be $\bb{Q}$-generic over $V$. In $V[G]$, let $\vec{\mathcal{C}} = \bigcup G$, 
  and let $\bb{T} = \bb{T}(\vec{\mathcal{C}})$. In $V$, 
  by Proposition \ref{dist_prop_2}, $\bb{Q} * \dot{\bb{T}}^\theta$ has a dense $\lambda$-directed closed subset 
  for all $0 < \theta < \lambda$. Let $H$ be $\bb{T}$-generic over $V[G]$.

	Our desired model is $V[G]$. Note that $\lambda$ remains inaccessible in $V[G]$ and, since 
  $\square(\lambda, < \lambda)$ holds, $\lambda$ is not weakly compact. Suppose for sake of contradiction 
  that $S = \langle I \times \kappa, \{R\} \rangle$ is a $\nu$-system in $V[G]$ such that 
  $\kappa < \lambda \leq \nu$, $\nu$ is regular, and $S$ has no cofinal branch. In $V[G*H]$, $\lambda$ 
  is supercompact. Fix $j:V[G*H] \rightarrow M$ witnessing that $\lambda$ is $\nu$-supercompact. 
  Let $\eta = \sup(j``\nu)$, let $\xi = \min(j(I) \setminus \eta)$ and, for $\gamma < \kappa$, let 
  $b_\gamma = \{u \in S \mid j(u) <_{j(R)} (\xi, \gamma)\}$. Again, it is easily verified, exactly 
  as in the proof of Theorem \ref{indestthm}, that $\bar{b} = \{b_\gamma \mid \gamma < \kappa\}$ is a full 
  set of branches through $S$. $\bar{b} \in V[G*H]$ and, in $V[G]$, $\bb{T}^{\kappa^+}$ is 
  $\lambda$-distributive, so Lemma \ref{neemanlemma} implies that $S$ has a cofinal branch in $V[G]$.
\end{proof}

\section{Counterexamples to the narrow system property} \label{counterexampleSect}

In this section, we construct narrow $\lambda$-systems with no cofinal branches from certain subadditive functions with domain $[\lambda]^2$. We then show that the existence of such functions follows from modified square principles.

\begin{proposition} \label{subadditiveProp}
	Suppose $\kappa < \lambda$ are cardinals, with $\lambda$ regular, and suppose there is $d:[\lambda]^2 \rightarrow \kappa$ satisfying:
	\begin{enumerate}
		\item{for all $\alpha < \beta < \gamma < \lambda$, $d(\alpha, \gamma) \leq \max(d(\alpha, \beta), d(\beta, \gamma))$;}
		\item{for all $\alpha < \beta < \gamma < \lambda$, $d(\alpha, \beta) \leq \max(d(\alpha, \gamma), d(\beta, \gamma))$;}
		\item{for all unbounded $I \subseteq \lambda$, $d``[I]^2$ is unbounded in $\kappa$.}
	\end{enumerate}
	Then there is a $\lambda$-system $S = \langle \lambda \times \kappa, \{R\} \rangle$ with no cofinal branch.
\end{proposition}

\begin{proof}
	To define the $\lambda$-system $S = \langle \lambda \times \kappa, \{R\} \rangle$, we only need to 
  specify the relation $R$. Given $\alpha_0 < \alpha_1 < \lambda$ and $\beta_0, \beta_1 < \kappa$, let 
  $(\alpha_0, \beta_0) <_R (\alpha_1, \beta_1)$ iff $\beta_0 = \beta_1 \geq d(\{\alpha_0, \alpha_1\})$. 
  It is simple to check that $S$ as defined is a $\lambda$-system; the fact that $R$ is transitive 
  follows from property (1) of $d$, and the fact that $R$ is tree-like follows from property (2) of $d$. 
  The fact that $S$ has no cofinal branch follows from property (3) of $d$.
\end{proof}

\begin{remark}
	We call a function $d:[\lambda]^2 \rightarrow \kappa$ satisfying (1) and (2) from the statement of Proposition \ref{subadditiveProp} \emph{subadditive}. We call a function satisfying (3) \emph{unbounded}.
\end{remark}

We now introduce two different modifications of $\square(\lambda)$ and show that each implies the existence of such subadditive functions. The first is a variant of \emph{indexed square}, a notion studied in \cite{cs} and \cite{cfm}.

\begin{definition} \label{ind_square_def}
	Let $\kappa < \lambda$ be infinite regular cardinals. A $\square^{\mathrm{ind}}(\lambda, \kappa)$-sequence is a matrix $\vec{\mathcal{C}} = \langle C_{\alpha,i} \mid \alpha < \lambda, i(\alpha) \leq i < \kappa \rangle$ satisfying the following conditions.
	\begin{enumerate}
		\item{For all limit $\alpha < \lambda$, $i(\alpha) < \kappa$.}
		\item{For all limit $\alpha < \lambda$ and $i(\alpha) \leq i < \kappa$, $C_{\alpha,i}$ is a club in $\alpha$.}
		\item{For all limit $\alpha < \lambda$ and $i(\alpha) \leq i < j < \kappa$, $C_{\alpha,i} \subseteq C_{\alpha,j}$.}
		\item{For all limit $\alpha < \beta < \lambda$ and $i(\beta) \leq i < \kappa$, if $\alpha \in C'_{\beta,i}$, then $i(\alpha) \leq i$ and $C_{\beta,i} \cap \alpha = C_{\alpha,i}$.}
		\item{For all limit $\alpha < \beta < \lambda$, there is $i < \kappa$ such that $\alpha \in C'_{\beta,i}$ (and hence $\alpha \in C'_{\beta,j}$ for all $i \leq j < \kappa$).}
		\item{There is no club $D \subseteq \lambda$ such that, for all $\alpha \in D'$, there is $i < \kappa$ such that $D \cap \alpha = C_{\alpha, i}$. (Such a club $D$ would be called a \emph{thread through $\vec{\mathcal{C}}$}.)}
	\end{enumerate}
	$\square^{\mathrm{ind}}(\lambda, \kappa)$ is the assertion that there is a $\square^{\mathrm{ind}}(\lambda, \kappa)$-sequence.
\end{definition}

\begin{proposition} \label{threadprop}
  Definition \ref{ind_square_def} is unchanged if we replace condition (6) by the following seemingly weaker condition:
	\[
	\mbox{There is no club } D \subseteq \lambda \mbox{ and } i < \kappa \mbox{ such that, for all } \alpha \in D', D \cap \alpha = C_{\alpha,i}.
	\]
\end{proposition}

\begin{proof}
	Suppose $\vec{\mathcal{C}} = \langle C_{\alpha, i} \mid \alpha < \lambda, i(\alpha) \leq i < \kappa \rangle$ satisfies conditions (1)-(5) of Definition 
  \ref{ind_square_def} and there is a club $D \subseteq \lambda$ such that, for every $\alpha \in D'$, there is $k < \kappa$ such that $D \cap \alpha = C_{\alpha,k}$. For each $\alpha \in D'$, let $k(\alpha)$ be such a $k$. Since $\kappa < \lambda$ and $\lambda$ is regular, there is an unbounded $A \subseteq D'$ and a $k^* < \kappa$ such that, for all $\alpha \in A$, $k(\alpha) = k^*$. We claim that, for all $\alpha \in D'$, $D \cap \alpha = C_{\alpha, k^*}$. To see this, fix $\alpha \in D'$, and find $\beta \in A \setminus (\alpha + 1)$. Then $D \cap \beta = C_{\beta, k^*}$, so, since $\alpha < \beta$ and 
  $\alpha \in D'$, we also have $\alpha \in C'_{\beta, k^*}$. Thus, by condition (4) of Definition \ref{ind_square_def}, 
  $D \cap \alpha = (D \cap \beta) \cap \alpha = C_{\beta, k^*} \cap \alpha = C_{\alpha, k^*}$.
\end{proof}

We will deal with the consistency of $\square^{\mathrm{ind}}(\lambda, \kappa)$ in Section \ref{indexedSquareSect}. 
We show now that it easily gives rise to subadditive, unbounded functions.

\begin{proposition}
	Suppose $\kappa < \lambda$ are infinite regular cardinals and $\square^{\mathrm{ind}}(\lambda, \kappa)$ holds. Then there is a subadditive, unbounded function $d:[\lambda]^2 \rightarrow \kappa$.
\end{proposition}

\begin{proof}
	Let $\pi$ be the unique order-preserving bijection from $\lambda$ to $\mathrm{lim}(\lambda)$. 
  For all $\alpha < \beta < \lambda$, let $d(\alpha, \beta)$ be the least $i < \kappa$ such that 
  $\pi(\alpha) \in C'_{\pi(\beta),i}$. By condition (5) of Definition \ref{ind_square_def}, $d$ is well-defined. 
  We now verify that $d$ satisfies conditions (1)-(3) from the statement of Proposition \ref{subadditiveProp}. 
  Note first that, if $\alpha < \beta < \lambda$ and $d(\alpha, \beta) \leq i < \kappa$, then $\pi(\alpha) \in 
  C'_{\pi(\beta), i}$, so, by condition (4) of Definition \ref{ind_square_def}, $C_{\pi(\beta), i} \cap 
  \alpha = C_{\pi(\alpha), i}$.

	To see (1), let $\alpha < \beta < \gamma < \lambda$, and suppose $i \geq \max(d(\alpha,\beta), d(\beta, \gamma))$. 
  Since $i \geq d(\alpha, \beta)$, $\pi(\alpha) \in C'_{\pi(\beta), i}$ and, since $i \geq d(\beta, \gamma)$, 
  $C_{\pi(\gamma), i} \cap \pi(\beta) = C_{\pi(\beta),i}$. Thus, $\pi(\alpha) \in C'_{\pi(\gamma), i}$, so 
  $i \geq d(\alpha, \gamma)$. 
  
  To see (2), let $\alpha < \beta < \gamma < \lambda$, and suppose $i \geq \max(d(\alpha, \gamma), 
  d(\beta, \gamma))$. As above, this implies that $\pi(\alpha) \in C'_{\pi(\gamma),i}$ and $C_{\pi(\gamma), i} 
  \cap \pi(\beta) = C_{\pi(\beta),i}$, so $\pi(\alpha) \in C'_{\pi(\beta), i}$ and hence $i \geq d(\alpha, \beta)$.

	To check (3), suppose for sake of contradiction that $I \subseteq \lambda$ is unbounded and 
  $j < \kappa$ is such that $d``[I]^2 \subseteq j$. Let 
  $D = \bigcup_{\beta \in I} C_{\pi(\beta), j}$. Note that, if $\beta < \gamma < \lambda$ and $\beta, \gamma \in I$, 
  we have $d(\beta, \gamma) \leq j$, so $C_{\pi(\gamma), j} \cap \pi(\beta) = C_{\pi(\beta), j}$. This implies that, for all $\beta \in I$, 
  $D \cap \pi(\beta) = C_{\pi(\beta), j}$. Therefore, since $I$ is unbounded in $\lambda$ and, for all $\beta \in I$, 
  $C_{\pi(\beta), j}$ is club in $\pi(\beta)$, $D$ is club in $\lambda$. Let $\alpha < \lambda$ be such 
  that $\pi(\alpha) \in D'$, and let $\beta \in I \setminus \alpha$. Then $D \cap \pi(\beta) = C_{\pi(\beta), j}$, 
  so $\pi(\alpha) \in C'_{\pi(\beta),j}$ and hence $D \cap \pi(\alpha) = C_{\pi(\beta), j} \cap \pi(\alpha) = 
  C_{\pi(\alpha), j}$. Thus, $D$ is a thread through $\vec{\mathcal{C}}$, which is a contradiction.
\end{proof}

The second square variation we consider is one in which we make slight additional demands on the order types of the clubs.

\begin{definition}
	Let $\kappa < \lambda$ be infinite regular cardinals. A $\square^\kappa(\lambda)$-sequence is a $\square(\lambda)$-sequence $\vec{C} = \langle C_\alpha \mid \alpha < \lambda \rangle$ such that, for stationarily many $\alpha \in S^\lambda_\kappa$, $\mathrm{otp}(C_\alpha) < \alpha$. $\square^\kappa(\lambda)$ is the assertion that there is a $\square^\kappa(\lambda)$-sequence.
\end{definition}

The proofs of Propositions 29-31 in \cite{lh} yield the following results.

\begin{proposition} \label{otprop}
	Let $\kappa < \lambda$ be infinite regular cardinals. The following are equivalent:
	\begin{enumerate}
		\item{$\square^\kappa(\lambda)$ holds;}
		\item{there is a $\square(\lambda)$-sequence $\vec{C} = \langle C_\alpha \mid \alpha < \lambda \rangle$ and a stationary set $S \subseteq S^\lambda_\kappa$ such that, for all $\alpha \in S$:
		\begin{itemize}
			\item{$\mathrm{otp}(C_\alpha) = \kappa$;}
			\item{for all $\beta$ such that $\alpha < \beta < \lambda$, $\alpha \not\in C'_\beta$.}
		\end{itemize}}
	\end{enumerate}
\end{proposition}

\begin{proposition}
	If $\theta < \kappa < \lambda$ are infinite regular cardinals, then $\square^\kappa(\lambda) \Rightarrow \square^\theta(\lambda)$.
\end{proposition}

We now introduce a function, due to Todor\v{c}evi\'{c} (cf. \cite{todorvcevic}, Section 7.2), that can be derived from a $\square(\lambda)$-sequence. 
In what follows, if $\alpha$ and $\beta$ are ordinals, we say that $\alpha$ \emph{divides} $\beta$ if there is an 
ordinal $\gamma$ such that $\beta = \alpha \cdot \gamma$. First, define $\Lambda_\kappa : [\lambda]^2 \rightarrow \lambda$ by 
\[\Lambda_\kappa (\alpha, \beta) = \max (\{ \xi \in C_\beta \cap (\alpha + 1) \mid \kappa \mbox{ divides } \mathrm{otp}(C_\beta \cap \xi) \}) \]
Next, let $\rho_\kappa : [\lambda]^2 \rightarrow \kappa$ be defined recursively by
\begin{multline*}
\rho_\kappa (\alpha, \beta) = \sup (\{\mathrm{otp}(C_\beta \cap [\Lambda_\kappa(\alpha, \beta),\alpha)), \rho_\kappa(\alpha, \min(C_\beta \setminus \alpha)), \\ \rho_\kappa(\xi, \alpha) \mid \xi \in C_\beta \cap [\Lambda_\kappa(\alpha, \beta), \alpha)\}).
\end{multline*}

The proof of the following proposition can be found in \cite{todorvcevic}.

\begin{proposition}[\cite{todorvcevic}, Lemma 7.2.2 and Theorem 7.2.13] \label{todorprop}
	Let $\kappa < \lambda$ be infinite regular cardinals, let $\vec{C}$ be a $\square(\lambda)$-sequence, and let $\rho_\kappa$ be derived as above from $\vec{C}$.
	\begin{enumerate}
		\item{For all $\alpha < \beta < \gamma < \lambda$, $\rho_\kappa(\alpha, \gamma) \leq \max(\rho_\kappa(\alpha, \beta), \rho_\kappa(\beta, \gamma))$.}
		\item{For all $\alpha < \beta < \gamma < \lambda$, $\rho_\kappa(\alpha, \beta) \leq \max(\rho_\kappa(\alpha, \gamma), \rho_\kappa(\beta, \gamma))$.}
		\item{If $\vec{C}$ is as in (2) of Proposition \ref{otprop}, then, for all unbounded $I \subseteq \lambda$, $\rho_\kappa ``[I]^2$ is unbounded in $\kappa$.}
	\end{enumerate}
\end{proposition}

Thus, if $\vec{C}$ is as in (2) of Proposition \ref{otprop} and $\rho_\kappa$ is derived from $\vec{C}$, then $\rho_\kappa$ is subadditive 
and unbounded.

\section{Forcing indexed square} \label{indexedSquareSect}

In this section, we demonstrate how to force the existence of a $\square^{\mathrm{ind}}(\lambda, \kappa)$-sequence.

\begin{definition}
	Let $\kappa < \lambda$ be infinite regular cardinals. $\bb{P}(\lambda, \kappa)$ is a forcing poset with conditions $p = \langle C^p_{\alpha, i} \mid \alpha \leq \gamma^p, i(\alpha)^p \leq i < \kappa \rangle$ satisfying the following conditions.
	\begin{enumerate}
		\item{$\gamma^p < \lambda$ is a limit ordinal and, for all limit $\alpha \leq \gamma^p$, $i(\alpha)^p < \kappa$.}
		\item{For all limit $\alpha \leq \gamma^p$ and $i(\alpha)^p \leq i < \kappa$, $C^p_{\alpha, i}$ is a club in $\alpha$.}
		\item{For all limit $\alpha \leq \gamma^p$ and $i(\alpha)^p \leq i < j < \kappa$, $C^p_{\alpha, i} \subseteq C^p_{\alpha, j}$.}
		\item{For all limit $\alpha < \beta \leq \gamma^p$ and $i(\beta)^p \leq i < \kappa$, if $\alpha \in (C^p_{\beta, i})'$, then $i(\alpha)^p \leq i$ and $C^p_{\beta, i} \cap \alpha = C^p_{\alpha, i}$.}
		\item{For all limit $\alpha < \beta \leq \gamma^p$, there is $i < \kappa$ such that $\alpha \in (C^p_{\beta, i})'$.}
	\end{enumerate}
	If $p,q \in \bb{P}(\lambda, \kappa)$, then $q \leq p$ iff $q$ end-extends $p$, i.e.:
	\begin{itemize}
		\item{$\gamma^q \geq \gamma^p$.}
		\item{For all limit $\alpha \leq \gamma^p$, $i(\alpha)^q = i(\alpha)^p$.}
		\item{For all limit $\alpha \leq \gamma^p$ and $i(\alpha)^p \leq i < \kappa$, $C^q_{\alpha, i} = C^p_{\alpha, i}$.}
	\end{itemize}
\end{definition}

\begin{proposition}
	$\bb{P}(\lambda, \kappa)$ is $\kappa$-directed closed.
\end{proposition}

\begin{proof}
	Let $\bb{P} = \bb{P}(\lambda, \kappa)$. Since $(\bb{P}, \geq)$ is tree-like, it suffices to verify that $\bb{P}$ is $\kappa$-closed. To this end, let $\xi < \kappa$ 
  be a limit ordinal, and let $\vec{p} = \langle p_\eta \mid \eta < \xi \rangle$ be a strictly decreasing sequence from $\bb{P}$. 

	Let $\gamma = \sup(\{\gamma^{p_\eta} \mid \eta < \xi \})$. We will define $q$, a lower bound for $\vec{p}$, so that $q = \langle C^q_{\alpha, i} \mid \alpha \leq \gamma, i(\alpha)^q \leq i < \kappa \rangle$. For limit $\alpha < \gamma$, let $\eta <\xi$ be least such that $\alpha \leq \gamma^{p_\eta}$, let $i(\alpha)^q = i(\alpha)^{p_\eta}$ and, for $i(\alpha)^q \leq i < \kappa$, let $C^q_{\alpha, i} = C^{p_\eta}_{\alpha, i}$. It remains to define $i(\gamma)^q$ and $C^q_{\gamma, i}$ for $i(\gamma)^q \leq i < \kappa$.

	For $\eta < \zeta < \xi$, let $i(\eta, \zeta)$ be the least $i < \kappa$ such that $\gamma^{p_\eta} \in (C^{p_\zeta}_{\gamma^{p_\zeta}, i})'$. Let $i^* = \sup(\{i(\eta, \zeta) \mid \eta < \zeta < \xi \})$. Since $\xi < \kappa$ and $\kappa$ is regular, $i^* < \kappa$. Also, for all $i^* \leq i < \kappa$ and all $\eta < \zeta < \xi$, 
  we have $\gamma^{p_\eta} \in (C^{p_\zeta}_{\gamma^{p_\zeta}, i(\eta, \zeta)})' \subseteq (C^{p_\zeta}_{\gamma^{p_\zeta}, i})'$, 
  so $C^q_{\gamma^{p_\zeta}, i} \cap \gamma^{p_\eta} = C^q_{\gamma^{p_\eta}, i}$. Thus, letting $i(\gamma)^q = i^*$ and, for all $i^* \leq i < \kappa$, $C^q_{\gamma, i} = \bigcup_{\eta < \xi} C^q_{\gamma^{p_\eta}, i}$, it is easily verified that $q \in \bb{P}$ and is a lower bound for $\vec{p}$.
\end{proof}

\begin{proposition} \label{scprop}
	$\bb{P}(\lambda, \kappa)$ is $\lambda$-strategically closed.
\end{proposition}

\begin{proof}
	Let $\bb{P} = \bb{P}(\lambda, \kappa)$. We describe a winning strategy for II in $G_\lambda(\bb{P})$. Suppose $0 < \xi < \lambda$ is an even ordinal and $\langle p_\eta \mid \eta < \xi \rangle$ is a partial play of $G_\lambda(\bb{P})$. Assume we have arranged inductively that, for all even $0 < \eta < \xi$, $i(\gamma^{p_\eta})^{p_\eta} = 0$ and, for all even $0 < \eta_0 < \eta$, $\eta_0 \in (C^{p_\eta}_{\gamma^{p_\eta}, 0})'$.

  Suppose first that $\xi = \eta + 2$ for some even $\eta < \lambda$. Let $\gamma = \gamma^{p_{\eta + 1}} + \omega$. We will define $p_\xi \leq p_{\eta + 1}$ so that $p_\xi = \langle C^{p_\xi}_{\alpha, i} \mid \alpha \leq \gamma, i(\alpha)^{p_\xi} \leq i < \kappa \rangle$. Since $p_\xi$ must be an end-extension of $p_{\eta + 1}$, we need only define $i(\gamma)^{p_\xi}$ and $C^{p_\xi}_{\gamma, i}$ for $i(\gamma)^{p_\xi} \leq i < \kappa$. As required to maintain the inductive hypothesis, we set $i(\gamma)^{p_\xi} = 0$. 
  We now define $C^{p_\xi}_{\gamma, i}$ for $i < \kappa$. The definition of $C^{p_\xi}_{\gamma, i}$ will depend upon 
  whether or not $\gamma^{p_\eta} \in (C^{p_{\eta + 1}}_{\gamma^{p_{\eta+1}}, i})'$. For $i < \kappa$ such that $\gamma^{p_\eta} \not\in (C^{p_{\eta + 1}}_{\gamma^{p_{\eta+1}}, i})'$, we let $C^{p_\xi}_{\gamma, i} = C^{p_\eta}_{\gamma^{p_\eta}, i} \cup \{\gamma^{p_\eta}\} \cup \{\gamma^{p_{\eta+1}} + n \mid n < \omega \}$. For $i < \kappa$ such that $\gamma^{p_\eta} \in (C^{p_{\eta + 1}}_{\gamma^{p_{\eta + 1}}, i})'$, we let $C^{p_\xi}_{\gamma, i} = C^{p_{\eta + 1}}_{\gamma^{p_{\eta + 1}}, i} \cup \{\gamma^{p_{\eta + 1}} + n \mid n < \omega \}$. It is easily verified that $p_\xi \leq p_{\eta + 1}$ and $i(\gamma^{p_\xi}) = 0$. Suppose that $\eta_0 < \xi$ is even. If $\eta_0 = \eta$, then, by construction, 
  $\eta_0 \in (C^{p_\xi}_{\gamma^{p_\xi}, 0})'$. If $\eta_0 < \eta$, then, by the inductive hypothesis applied to $\eta$, we have $\eta_0 \in 
  (C^{p_\eta}_{\gamma^{p_\eta}, 0})' \subseteq (C^{p_\xi}_{\gamma^{p_\xi}, 0})'$, so we have maintained the inductive hypothesis.

	Next, suppose that $\xi$ is a limit ordinal. Let $\gamma = \sup(\{\gamma^{p_\eta} \mid \eta < \xi \})$. We define $p_\xi$ to be a lower bound for $\langle p_\eta \mid \eta < \xi \rangle$ of the form $\langle C^{p_\xi}_{\alpha, i} \mid \alpha \leq \gamma, i(\alpha)^{p_\xi} \leq i < \kappa \rangle$. Again, we only have to specify $i(\gamma)^{p_\xi}$ and $C^{p_\xi}_{\gamma, i}$ for $i(\gamma)^{p_\xi} \leq i < \kappa$. Let $i(\gamma)^{p_\xi} = 0$ and, for all $i < \kappa$, let $C^{p_\xi}_{\gamma, i} = \bigcup \{C^{p_\eta}_{\gamma^{p_\eta}, i} \mid \eta < \xi, \eta$ even$\}$. It is again easily verified that $p_\xi$ is a lower bound for $\langle p_\eta \mid \eta < \xi \rangle$ and maintains the inductive hypothesis.
\end{proof}

\begin{corollary}
	Forcing with $\bb{P}(\lambda, \kappa)$ preserves all cardinalities and cofinalities $\leq \lambda$. If, in addition, $\lambda^{< \lambda} = \lambda$, then $|\bb{P}(\lambda, \kappa)| = \lambda$ and hence preserves all cardinalities and cofinalities.
\end{corollary}

A variation on the proof of Proposition \ref{scprop} yields the following.

\begin{proposition}
	Let $\alpha < \lambda$, and let $D_\alpha = \{p \in \bb{P}(\lambda, \kappa) \mid \alpha \leq \gamma^{p} \}$. Then $D_\alpha$ is dense in $\bb{P}(\lambda, \kappa)$.
\end{proposition}

\begin{proposition} \label{nothreadprop}
	Let $G$ be $\bb{P}(\lambda, \kappa)$-generic over $V$. Let $\vec{\mathcal{C}} = \bigcup G = \langle C_{\alpha, i} \mid \alpha < \lambda, i(\alpha) \leq i < \kappa \rangle$. Then $\vec{\mathcal{C}}$ is a $\square^{\mathrm{ind}}(\lambda, \kappa)$-sequence.
\end{proposition}

\begin{proof}
  Let $\bb{P} = \bb{P}(\lambda, \kappa)$. The fact that $\vec{\mathcal{C}}$ satisfies conditions (1)-(5) in Definition \ref{ind_square_def} 
  follows easily from the definition of $\bb{P}$. Thus, we only need to verify that $\vec{\mathcal{C}}$ satisfies condition (6) or, 
  alternatively, the condition identified in Proposition \ref{threadprop}. To this end, suppose for sake of contradiction that there 
  is a club $D \subseteq \lambda$ and an $i^* < \kappa$ such that, for all $\alpha \in D'$, $D \cap \alpha = C_{\alpha, i^*}$.

	For each $\alpha < \lambda$ and $i < \kappa$, let $\dot{C}_{\alpha, i}$ be a canonical name for $C_{\alpha,i}$ (where, if $p \in \bb{P}$ and $p$ decides the value of $i(\alpha)$, then $p \Vdash ``\dot{C}_{\alpha, i} = \emptyset"$ for all $i < i(\alpha)$). Find $p \in G$ and a $\bb{P}$-name $\dot{D}$ such that $p \Vdash ``\dot{D}$ is club in $\lambda$ and, for all $\alpha \in \dot{D}', \dot{D} \cap \alpha = \dot{C}_{\alpha, i^*}$". Working in $V$, we play a run of $G_\omega(\bb{P})$, $\langle p_n \mid n < \omega \rangle$ with II playing according to the winning strategy described in Proposition \ref{scprop} and I playing to ensure that the following hold:
	\begin{itemize}
		\item{$p_1 \leq p$;}
		\item{there is a strictly increasing sequence of ordinals $\langle \beta_n \mid n < \omega \rangle$ such that:
		\begin{itemize}
			\item{for all $n < \omega$, $\gamma^{p_{2n}} \leq \beta_n < \gamma^{p_{2n+1}}$ (where we assign $\gamma^{p_0} = 0$);}
			\item{for all $n < \omega$, $p_{2n+1} \Vdash ``\beta_n \in \dot{D}"$.}
		\end{itemize}}
	\end{itemize}

  Let $\gamma = \sup(\{\gamma^{p_n} \mid n < \omega \}) = \sup(\{\beta_n \mid n < \omega\})$. We will define $q \in \bb{P}$, a lower bound for $\langle p_n \mid n < \omega \rangle$, of the form $\langle C^q_{\alpha, i} \mid \alpha \leq \gamma, i(\alpha)^q \leq i < \kappa \rangle$. As usual, we only need to specify $i(\gamma)^q$ and $C^q_{\gamma, i}$ for all $i(\gamma)^q \leq i < \kappa$. Let $i(\gamma)^q = i^* + 1$. Note that, as II played according to the winning strategy described in Proposition \ref{scprop}, the following hold.
  \begin{itemize}
    \item For all $n < \omega$, $i(\gamma^{p_{2n}})^{p_{2n}} = 0$.
    \item For all $m < n < \omega$, $\gamma^{p_{2m}} \in (C^{p_{2n}}_{\gamma^{p_{2n}}, 0})'$.
  \end{itemize}
  Therefore, if, for all $i(\gamma)^q \leq i < \kappa$, we let $C^q_{\gamma, i} = \bigcup \{C^{p_{2n}}_{\gamma^{p_{2n}}, i} \mid n < \omega \}$, then it is easily verified that $q$ is a lower bound for $\langle p_n \mid n < \omega \rangle$. Since $q \leq p$, $q \Vdash ``\dot{D}$ is a club,$"$ so, as $q$ is a lower bound for $\langle p_n \mid n < \omega \rangle$, $q \Vdash ``\gamma \in \dot{D}'."$ Thus, again because $q \leq p$, $q \Vdash ``\dot{D} \cap \gamma = \dot{C}_{\gamma, i^*}."$ However, as $i(\gamma)^q = i^* + 1$, $q \Vdash ``\dot{C}_{\gamma, i^*} = \emptyset,"$ which is a contradiction.
\end{proof}

\section{Separating squares} \label{separatingSect}

In \cite{lh}, we proved that, if $\lambda$ is the successor of a regular cardinal and $\kappa < \nu < \lambda$ are regular cardinals, it is not necessarily the case that $\square^{\kappa}(\lambda) \Rightarrow \square^{\nu}(\lambda)$. In this section, we prove an analogous result when $\lambda$ is the successor of a singular cardinal. We consider the case $\lambda = \aleph_{\omega+1}$, but the same technique will work for other successors of singular cardinals. We will need the following, a proof of which can be found in \cite{lh}.

\begin{proposition} \label{genericity_prop}
	Let $\kappa < \lambda$ be infinite, regular cardinals, and let $\bb{Q} = \bb{Q}(\lambda, 1)$ be the forcing to add a $\square(\lambda)$-sequence. If $G$ is $\bb{Q}$-generic over $V$, then $\bigcup G$ is a $\square^\kappa(\lambda)$-sequence. 
\end{proposition}

\begin{theorem}
	Suppose $\kappa$ is a supercompact cardinal, GCH holds, and $n < \omega$. Let $\lambda = \kappa^{+\omega + 1}$. Then there is a forcing extension in which all cardinals $\leq \aleph_{n+1}$ are preserved, $\kappa = \aleph_{n+2}$, $\lambda = \aleph_{\omega + 1}$, $\square^{\aleph_n}(\lambda)$ holds, and $\square^{\aleph_{n+1}}(\lambda)$ fails.
\end{theorem}

\begin{proof}
	The proof follows that of Theorem 3.15 in \cite{lh}. We thus omit many of the details and refer the reader to the earlier paper.

	Let the initial model be called $V_0$. In $V_0$, let $\bb{P} = \mathrm{Coll}(\aleph_{n+1}, < \kappa)$. Let $G$ be $\bb{P}$-generic 
  over $V_0$, and let $V = V_0[G]$. Work for now in $V$. Let $\bb{Q} = \bb{Q}(\lambda, 1)$, and let $\dot{\vec{C}}$ be the 
  canonical $\bb{Q}$-name for the $\square(\lambda)$-sequence added by $\bb{Q}$. Let $\dot{\bb{T}}$ be a $\bb{Q}$-name for $\bb{T}(\dot{\vec{C}})$.

  Working in $V^{\bb{Q}}$, we define a sequence of posets $\langle \bb{S}_\alpha \mid \alpha \leq \lambda^+ \rangle$ by 
  induction on $\alpha$. Each $\bb{S}_\alpha$ will be $\lambda$-distributive and so will preserve the cardinal structure 
  below $\lambda$. For each $\beta < \lambda^+$, we will fix an $\bb{S}_\beta$-name $\dot{X}_\beta$ for a subset of 
  $S^{\lambda}_{\aleph_{n+1}}$ such that $\Vdash_{\bb{S}_\beta * \dot{\bb{T}}} ``\dot{X}_\beta$ is non-stationary.$"$ 
  If $\alpha \leq \lambda^+$, then elements of $\bb{S}_\alpha$ are functions $s$ such that:
	\begin{enumerate}
		\item{$\mathrm{dom}(s) \subseteq \alpha$;}
		\item{$|s| \leq \kappa^{+\omega}$;}
		\item{for every $\beta \in \mathrm{dom}(s)$, $s(\beta)$ is a closed, bounded subset of $\lambda$;}
		\item{for every $\beta \in \mathrm{dom}(s)$, $s \restriction \beta \Vdash_{\bb{S}_\beta} ``s(\beta) \cap \dot{X}_\beta = \emptyset."$}
	\end{enumerate}
  If $s,t \in \bb{S}_\alpha$, then $t \leq s$ iff $\mathrm{dom}(s) \subseteq \mathrm{dom}(t)$ and, for every 
  $\beta \in \mathrm{dom}(s)$, $t(\beta)$ end-extends $s(\beta)$. Let $\bb{S} = \bb{S}_{\lambda^+}$. It is easily 
  seen that, for every $\alpha < \lambda^+$, $\Vdash_{\bb{S}} ``\dot{X}_\alpha$ is non-stationary." Moreover, 
  by GCH and a standard $\Delta$-system argument, $\bb{S}$ has the $\lambda^+$-chain condition. Therefore, every 
  canonical $\bb{S}$-name (resp. $\bb{S} * \dot{\bb{T}}$-name) for a subset of $\lambda$ is an $\bb{S}_\alpha$-name 
  (resp. $\bb{S}_\alpha * \dot{\bb{T}}$-name) for some $\alpha < \lambda^+$, so, if $\dot{X}$ is a canonical $\bb{S}$-name for a subset 
  of $\lambda$ and $\Vdash_{\bb{S}*\dot{\bb{T}}}``\dot{X}$ is non-stationary,$"$ then there is $\alpha < \lambda^+$ 
  such that $\dot{X}$ is an $\bb{S}_\alpha$-name and $\Vdash_{\bb{S}_\alpha * \dot{\bb{T}}}``\dot{X}$ is non-stationary.$"$
  An easy counting argument shows that there are only $\lambda^+$ canonical $\bb{S}$-names for subsets of $\lambda$, 
  so we can choose the sequence $\langle \dot{X}_\alpha \mid \alpha < \lambda^+ \rangle$ in such a way that, if 
  $\dot{X}$ is a canonical $\bb{S}$-name for a subset of $S^{\lambda^+}_{\aleph_{n+1}}$ and $\Vdash_{\bb{S}*\dot{\bb{T}}}``\dot{X}$ is 
  non-stationary,$"$ then there is $\alpha < \lambda^+$ such that $\dot{X} = \dot{X}_\alpha$. In particular, we may arrange so that, 
  in $V^{\bb{Q} * \dot{\bb{S}}}$, if $X \subseteq S^{\lambda^+}_{\aleph_{n+1}}$ and $\Vdash_{\bb{T}}``\check{X}$ is non-stationary$,"$ 
  then $X$ is already non-stationary in $V^{\bb{Q}*\dot{\bb{S}}}$.
    
	The following lemmas are proved as in \cite{lh}.
	\begin{lemma}
		$\bb{S}$ is $\aleph_{n+1}$-closed.
	\end{lemma}

	\begin{lemma}
		In $V$, $\bb{Q} * \dot{\bb{S}} * \dot{\bb{T}}$ has a dense $\lambda$-directed closed subset.
	\end{lemma}

	$V^{\bb{Q} * \dot{\bb{S}}}$ will be our desired model. Thus, let $H$ be $\bb{Q}$-generic over $V$, and let 
  $I$ be $\bb{S}$-generic over $V[H]$. It is clear that, in $V[H*I]$, $\kappa = \aleph_{n+2}$ and 
  $\lambda = \aleph_{\omega + 1}$. We first argue that $\square^{\aleph_n}(\lambda)$ holds. Since, by Proposition 
  \ref{genericity_prop}, forcing with 
  $\bb{Q}$ adds a $\square^{\aleph_n}(\lambda)$-sequence, we know that $\square^{\aleph_n}(\lambda)$ holds in 
  $V[H]$. By Proposition \ref{otprop}, we can fix a $\square(\lambda)$-sequence $\vec{C} = \langle C_\alpha 
  \mid \alpha < \lambda \rangle$ and a stationary set $S \subseteq S^\lambda_{\aleph_n}$ such 
  that, for every $\alpha \in S$, $\mathrm{otp}(C_\alpha) = \aleph_n$.

	\begin{claim}
		In $V[H]$, $S \in I[\lambda]$.
	\end{claim}

	\begin{proof}
    Let $\vec{a} = \langle a_\alpha \mid \alpha < \lambda \rangle$ be an enumeration of all bounded subsets of 
    $\lambda$. Suppose first that $n = 0$, and let 
    \[
      E = \{\gamma < \lambda \mid\mbox{ for all }a \in [\gamma]^{<\omega}, 
      \mbox{ there is }\alpha < \gamma\mbox{ such that }a = a_\alpha\}.
    \]
    $E$ is easily seen to be a club in $\lambda$. 
    If $\gamma \in E \cap S$, $A$ is any $\omega$-sequence cofinal in $\gamma$, and $\beta < \gamma$, then 
    $A \cap \beta \in [\gamma]^{<\omega}$, so there is $\alpha < \gamma$ such that $A \cap \beta = a_\alpha$. 
    Therefore, $A$ witnesses that $\gamma$ is approachable with respect to $\vec{a}$.
    
    Next, assume $n > 0$. Define a function $f:[\lambda]^2 \rightarrow \lambda$ by letting, for $\alpha < \beta < \lambda$, 
    $f(\alpha, \beta)$ be such that $C_\beta \cap \alpha = a_{f(\alpha, \beta)}$. Let $E$ be the set of closure points of $f$ 
    below $\lambda$. Then $E$ is a club in $\lambda$. We claim that, if $\gamma \in S \cap E$, 
    then $\gamma$ is approachable with respect to $\vec{a}$. To see this, fix such a $\gamma$. Let $\alpha < \gamma$, 
    and let $\beta = \min(C'_\gamma \setminus (\alpha+1))$. We have $\alpha < \beta < \gamma$ and 
    $C_\gamma \cap \alpha = C_\beta \cap \alpha = a_{f(\alpha, \beta)}$. Since $\gamma$ is a closure point of $f$, 
    we have $f(\alpha, \beta) < \gamma$, so $C_\gamma$ witnesses that $\gamma$ is approachable with respect to $\vec{a}$.
	\end{proof}

	Since $S \in I[\lambda]$ and $\bb{S}$ is $\aleph_{n+1}$-closed, Fact \ref{apstat} implies that $S$ remains 
  stationary in $V[H*I]$. We claim that $\vec{C}$ remains a $\square^{\aleph_n}(\lambda)$-sequence in $V[H*I]$. 
  The only slightly non-trivial condition to check is the requirement that there is no thread through $\vec{C}$. 
  Suppose for sake of contradiction that $D$ is such a thread. Since $S$ is stationary, we can find 
  $\alpha < \beta$, both in $D' \cap S$. Since $D$ is a thread, this means that $D \cap \alpha = C_\alpha$ and 
  $D \cap \beta = C_\beta$. Since $\alpha, \beta \in S$, we have $\mathrm{otp}(C_\alpha) = \mathrm{otp}(C_\beta) = \aleph_n$, 
  contradicting the fact that $\mathrm{otp}(D \cap \alpha) < \mathrm{otp}(D \cap \beta)$.

	We finally show that $\square^{\aleph_{n+1}}(\lambda)$ fails in $V[H*I]$. Suppose on the contrary that, 
  in $V[H*I]$, $\vec{D} = \langle D_\alpha \mid \alpha < \lambda \rangle$ is a $\square(\lambda)$-sequence and 
  $T \subseteq S^\lambda_{\aleph_{n+1}}$ is stationary such that, for all $\alpha \in T$, $\mathrm{otp}(D_\alpha) < \alpha$. 
  By our construction of $\bb{S}$, we can let $J$ be $\bb{T}$-generic over $V[H*I]$ such that $T$ remains stationary 
  in $V[H*I*J]$.

	In $V_0$, let $j:V_0 \rightarrow M_0$ witness that $\kappa$ is $\lambda^+$-supercompact. 
  In $V_0^{\bb{P}}$, $\bb{Q} * \dot{\bb{S}} * \dot{\bb{T}}$ has a dense $\lambda$-directed 
  closed subset of size $\lambda^+$, so, by Fact \ref{lift}, the identity embedding from $\bb{P}$ into 
  $j(\bb{P}) = \mathrm{Coll}(\aleph_{n+1}, <j(\kappa))$ can be extended to a complete 
  embedding from $\bb{P} * \dot{\bb{Q}} * \dot{\bb{S}} * \dot{\bb{T}}$ into $j(\bb{P})$ in such a way 
  that the quotient forcing is $\aleph_{n+1}$-closed. Thus, letting $\dot{\bb{R}}$ be a 
  $\bb{P} * \dot{\bb{Q}} * \dot{\bb{S}} * \dot{\bb{T}}$-name for this quotient forcing and letting 
  $K$ be $\bb{R}$-generic over $V[H*I*J]$, we can, in $V[H*I*J*K]$, extend $j$ to $j:V \rightarrow M_0[G*H*I*J*K]$, 
  and denote $M_0[G*H*I*J*K]$ by $M$.
  Exactly as was done in the proof of Theorem \ref{squarenspthm}, using the fact that $j \restriction \bb{P} * \dot{\bb{Q}} 
  * \dot{\bb{S}} * \dot{\bb{T}} \in M_0$ and the fact that, in $V$, $\bb{P} * \dot{\bb{Q}} * \dot{\bb{S}}$ has a 
  dense $\lambda$-directed closed subset, we can find a master condition $(q^*, \dot{s}^*, \dot{t}^*)$ in 
  $j(\bb{Q} * \dot{\bb{S}} * \dot{\bb{T}})$, such that, letting $H^+ * I^+$ be $j(\bb{Q} * \dot{\bb{S}})$-generic 
  over $V[H*I*J*K]$ with $(q^*, \dot{s}^*) \in H^+ * I^+$, we have $j``H*I \subseteq H^+ * I^+$. Then, in $V[H*I*J*K*H^+*I^+]$, we can extend $j$ 
  to $j:V[H*I] \rightarrow M[H^+ * I^+]$.

	Let $\eta = \sup(j``\lambda) < j(\lambda)$. Note that, in $V[H*I*J*K*H^+*I^+]$, $\cf(\eta) = \aleph_{n+1}$ and 
  $j``\lambda$ is $(<\kappa)$-club in $\eta$. Let $j(\vec{D}) = \vec{E} = \langle E_\alpha \mid \alpha < j(\lambda) 
  \rangle$. Let $F = \{\alpha \in S^{\lambda}_{<\kappa} \mid j(\alpha) \in E'_\eta \}$. $F$ is $(<\kappa)$-club in 
  $\lambda$. If $\alpha < \beta$ are both in $F$, then $j(\alpha), j(\beta) \in E'_\eta$, so $j(D_\alpha) = 
  E_{j(\alpha)} = E_\eta \cap j(\alpha)$ and $j(D_\beta) = E_{j(\beta)} = E_\eta \cap j(\beta)$, so $j(D_\alpha) = 
  j(D_\beta) \cap j(\alpha)$. By elementarity, $D_\alpha = D_\beta \cap \alpha$, so $F^* = \bigcup_{\alpha \in F} 
  D_\alpha$ is a thread through $\vec{D}$. $F^* \in V[H*I*J*K*H^+*I^+]$.

	\begin{claim}
    There is a thread through $\vec{D}$ in $V[H*I*J]$.
	\end{claim}
	\begin{proof}
    Suppose not. Note that, since $\bb{P} * \dot{\bb{Q}} * \dot{\bb{S}} * \dot{\bb{T}} * \dot{\bb{R}}$ has an 
    $\omega_1$-closed dense subset and ${^{\lambda^+}}M_0 \subseteq M_0$ in $V_0$, we certainly have ${^\omega}M \subseteq M$ 
    in $V[H*I*J*K]$. Therefore, since $j(\bb{Q} * \dot{\bb{S}})$ is $\omega_1$-closed in $M$, it is $\omega_1$-closed 
    in $V[H*I*J*K]$ as well. Therefore, in $V[H*I*J]$, $\bb{W} := \bb{R} * j(\bb{Q} * \dot{\bb{S}})$ is $\omega_1$-closed.
    
    Work in $V[H*I*J]$. Let $\dot{F}$ be a $\bb{W}$-name for $F^*$, and let $v \in \bb{W}$ force that $\dot{F}$ is a thread through $\vec{D}$. 

    \begin{subclaim}
      There are $w^0, w^1 \leq v$ and $\alpha < \lambda$ such that $w^0 \Vdash ``\check{\alpha} \in \dot{F}"$ and 
      $w^1 \Vdash ``\check{\alpha} \not\in \dot{F}."$
    \end{subclaim}

    \begin{proof}
      Suppose not, and let $Z = \{\alpha < \lambda \mid$ for some $w \leq v, w \Vdash ``\check{\alpha} \in \dot{F}"\}$. 
      Then $v \Vdash ``\dot{F} = \check{Z},"$ so, since $v \Vdash ``\dot{F}$ is a thread through $\vec{D},"$ it must 
      be the case that $Z$ is a thread through $\vec{D}$. However, $Z \in V[H*I*J]$, contradicting our assumptions.
    \end{proof}

    Fix $w^0, w^1 \leq v$ and $\alpha < \lambda$ as in the subclaim. Now, inductively define conditions $\langle w^i_m \mid i < 2, m < \omega \rangle$ 
    and ordinals $\langle \beta^i_m \mid i < 2, m < \omega \rangle$ such that:
		\begin{itemize}
			\item{for each $i < 2$, $\langle w^i_m \mid m < \omega \rangle$ is a decreasing sequence of conditions from $\bb{W}$ and $w^i_0 \leq w^i$;}
			\item{for all $m < \omega$, $\alpha < \beta^0_m < \beta^1_m < \beta^0_{m+1} < \lambda$;}
			\item{for all $i < 2$ and $m < \omega$, $w^i_m \Vdash ``\beta^i_m \in \dot{F}."$}
		\end{itemize}
		Let $\gamma = \sup(\{\beta^0_m \mid m < \omega\}) = \sup(\{\beta^1_m \mid m < \omega\})$ and, for $i < 2$, use the closure of $\bb{W}$ to find a 
    lower bound $w^i_\infty$ for $\langle w^i_m \mid m < \omega \rangle$. For each $i < 2$, $w^i_\infty \Vdash ``\gamma \in \dot{F}',"$ so 
    $w^i_\infty \Vdash ``\dot{F} \cap \gamma = C_\gamma."$ Then $w^0_\infty \leq w^0$ implies that $\alpha \in C_\gamma$, 
    while $w^1_\infty \leq w^1$ implies $\alpha \not\in C_\gamma$. Contradiction.
	\end{proof}

	Let $A$ be a thread through $\vec{D}$ in $V[H*I*J]$. Then $A^* = \{\alpha \in A \mid \mathrm{otp}(A \cap \alpha) = \alpha\}$ is club in $\lambda$, 
  since it is the set of closure points of the function $g:\lambda \rightarrow A$ defined by letting $g(\alpha)$ be 
  the unique $\beta \in A$ such that $\mathrm{otp}(A \cap \beta) = \alpha$. In addition, $A^* \cap T = \emptyset$, 
  contradicting the fact that $T$ is stationary in $V[H*I*J]$ and completing the proof.
\end{proof}

\section{Derived systems} \label{derivedSect}

One of the most useful properties of systems is that, when $\bb{P}$ is a sufficiently small forcing poset, 
a $\bb{P}$-name for a system gives rise to a system in the ground model. Let $\bb{P}$ be a forcing poset, 
let $\tau < \lambda$ be cardinals, with $\lambda$ regular and $|\bb{P}| < \lambda$, and suppose 
$\dot{S} = \langle \bigcup_{\alpha \in \dot{I}} \{\alpha\} \times \dot{\kappa}_\alpha, \{\dot{R}_i \mid i < \tau \} \rangle$ 
is a $\bb{P}$-name for a $\lambda$-system. Since every $\lambda$-system $\langle \bigcup_{\alpha \in I} \{\alpha\} \times \kappa_\alpha, 
\mathcal{R} \rangle$ is isomorphic to one in which $I = \lambda$, we may assume 
$\dot{I} = \check{\lambda}$. For each $\alpha < \lambda$, find $p_\alpha \in \bb{P}$ deciding the value of 
$\dot{\kappa}_\alpha$. Since $|\bb{P}| < \lambda$, there is a $p \in \bb{P}$ such that, for unboundedly many 
$\alpha < \lambda$, $p_\alpha = p$. Thus, by passing to a name for a subsystem and working below a condition in 
$\bb{P}$, we may assume $\dot{S}$ is of the form $\langle \bigcup_{\alpha < \lambda} \{\alpha\} \times \kappa_\alpha, 
\{\dot{R}_i \mid i < \tau \} \rangle$.  In $V$, we define the derived system $D_{\bb{P}}(\dot{S}) = 
\langle \bigcup_{\alpha < \lambda} \{\alpha\} \times \kappa_\alpha, \{R_{i,p} \mid i < \tau, p \in \bb{P} \} \rangle$ 
by letting, for every $\alpha_0 < \alpha_1 < \lambda$, $\beta_0 < \kappa_{\alpha_0}, \beta_1 < \kappa_{\alpha_1}$, 
$i < \tau$, and $p \in \bb{P}$, $(\alpha_0, \beta_0) <_{R_{i,p}} (\alpha_1, \beta_1)$ iff 
$p \Vdash ``(\alpha_0, \beta_0) <_{\dot{R}_i} (\alpha_1, \beta_1)."$ It is easily verified that $D_{\bb{P}}(\dot{S})$ 
is a $\lambda$-system and $\mathrm{width}(D_{\bb{P}}(\dot{S})) = \max(\{\sup(\{\kappa_\alpha \mid \alpha < \lambda\}), \tau, |\bb{P}|\})$.

\begin{proposition}
	Suppose $\bb{P}$ is a forcing poset, $\tau < \lambda$ are cardinals, with $\lambda$ regular, and 
  $\dot{S} = \langle \bigcup_{\alpha < \lambda} \{\alpha\} \times \kappa_\alpha, \{\dot{R}_i \mid i < \tau \} \rangle$ 
  is a $\bb{P}$-name for a $\lambda$-system. If $\Vdash_{\bb{P}}``\dot{S}$ has no cofinal branch,$"$ 
  then $D_{\bb{P}}(\dot{S})$ has no cofinal branch in $V$
\end{proposition}

\begin{proof}
  Suppose $i < \tau$, $p \in \bb{P}$, and $b$ is a cofinal branch through $R_{i,p}$ in $D_{\bb{P}}(\dot{S})$. 
  Then $p \Vdash ``\check{b}$ is a cofinal branch through $\dot{R}_i$ in $\dot{S}."$
\end{proof}

\begin{proposition}
	If $\mu \leq \lambda$, with $\lambda$ regular, then $NSP(\mu, \lambda)$ is indestructible under forcing with posets $\bb{P}$ such that $|\bb{P}| < \mu$ and $|\bb{P}|^+ < \lambda$.
\end{proposition}

\begin{proof}
	Suppose that $\mu \leq \lambda$, $NSP(\mu, \lambda)$ holds in $V$, $\bb{P}$ is a forcing poset, $|\bb{P}| < \mu$, and $|\bb{P}|^+ < \lambda$. Suppose for sake of contradiction that there is $p \in \bb{P}$ and a $\bb{P}$-name $\dot{S} = \langle \lambda \times \kappa, \{\dot{R}_i \mid i < \tau \} \rangle$ such that $\kappa, \tau < \mu$ and $p \Vdash_{\bb{P}} ``\dot{S}$ is a narrow $\lambda$-system with no cofinal branch." Let $\bb{Q} = \{q \in \bb{P} \mid q \leq p \}$. Re-interpreting $\dot{S}$ as a $\bb{Q}$-name, we obtain $\Vdash_{\bb{Q}}``\dot{S}$ is a narrow $\lambda$-system with no cofinal branch." Then, in $V$, $D_{\bb{Q}}(\dot{S})$ is a $\lambda$-system with no cofinal branch, $\mathrm{width}(D_{\bb{Q}}(\dot{S})) = \max(\{\kappa, \tau, |\bb{P}|\}) < \mu$, and $\mathrm{width}(D_{\bb{Q}}(\dot{S}))^+ < \lambda$, contradicting $NSP(\mu, \lambda)$. 
\end{proof}

\begin{proposition} \label{wide_system_prop}
	Let $\kappa$ be an infinite cardinal. There is a $\kappa^+$-system of width $\kappa$ with no cofinal branch.
\end{proposition}

\begin{proof}
	Let $\bb{P} = \mathrm{Coll}(\omega, \kappa)$. Then $|\bb{P}| = \kappa$ and, in $V^{\bb{P}}$, $\kappa^+ = \omega_1$. Thus, in $V^{\bb{P}}$, there is a $\kappa^+$-Aronszajn tree, $T$. We may assume that, for all $\alpha < \kappa^+$, level $\alpha$ of $T$ is the set $\{\alpha\} \times \omega$. $T$ can thus be thought of as a $\kappa^+$-system of width $\omega$ with one relation (the tree relation). Letting $\dot{T}$ be a $\bb{P}$-name for $T$, we may form the derived system $D_{\bb{P}}(\dot{T})$ in $V$. $D_{\bb{P}}(\dot{T})$ is a $\kappa^+$-system of width $|\bb{P}| = \kappa$. Since $\Vdash_{\bb{P}}``\dot{T}$ has no cofinal branch,$"$ $D_{\bb{P}}(\dot{T})$ has no cofinal branch in $V$.
\end{proof}

We now introduce another variation on Jensen's square principle.

\begin{definition}
	Suppose $\kappa \leq \mu$ are infinite cardinals, with $\kappa$ regular. $\langle C_\alpha \mid \alpha \in A \rangle$ is a $\square^{\geq \kappa}_\mu$-sequence if:
	\begin{enumerate}
		\item{$S^{\mu^+}_{\geq \kappa} \subseteq A \subseteq \mathrm{lim}(\mu^+)$;}
		\item{for all $\alpha \in A$, $C_\alpha$ is club in $\alpha$ and $\mathrm{otp}(C_\alpha) \leq \mu$;}
		\item{for all $\beta \in A$ and $\alpha < \beta$, if $\alpha \in C'_\beta$, then $\alpha \in A$ and $C_\beta \cap \alpha = C_\alpha$.}
	\end{enumerate}
	$\square^{\geq \kappa}_\mu$ is the assertion that a $\square^{\geq \kappa}_\mu$-sequence exists.
\end{definition}

Square principles of this sort were first studied by Baumgartner, in unpublished work. Let $\bb{B}(\mu, \kappa)$ be the forcing poset whose conditions are of the form $p = \langle C^p_\alpha \mid \alpha \in s^p \rangle$ such that:
\begin{itemize}
	\item{$s^p$ is a bounded subset of $\mu^+$ with a maximal element, $\gamma^p$.}
	\item{$\gamma^p \cap \mathrm{cof}(\geq \kappa) \subseteq s^p$.}
	\item{For all $\alpha \in s^p$, $C^p_\alpha$ is a club in $\alpha$ and $\mathrm{otp}(C_\alpha) \leq \mu$.}
	\item{For all $\beta \in s^p$ and $\alpha < \beta$, if $\alpha \in (C^p_\beta)'$, then $\alpha \in s^p$ and $C^p_\beta \cap \alpha = C^p_\alpha$.}
\end{itemize}
If $p,q \in \bb{B}(\mu, \kappa)$, then $q \leq p$ if $s^q$ end-extends $s^p$ and, for all $\alpha \in s^p$, $C^q_\alpha = C^p_\alpha$. 
The following is easily proven in the usual manner (see \cite[Section 2.2]{ac} for details).

\begin{proposition}
	Let $\kappa \leq \mu$ be infinite cardinals, with $\kappa$ regular.
	\begin{enumerate}
		\item{$\bb{B}(\mu, \kappa)$ is $\kappa$-directed closed.}
		\item{$\bb{B}(\mu, \kappa)$ is $\mu+1$-strategically closed.}
		\item{$\Vdash_{\bb{B}(\mu, \kappa)}``\square^{\geq \kappa}_\mu$ holds.$"$}
	\end{enumerate}
\end{proposition}

\begin{proposition} \label{baumgartnersquareprop}
	Let $\kappa < \mu$ be infinite cardinals, and suppose $\square^{\geq \kappa^+}_\mu$ holds. Then there is a $\mu^+$-system of width $\kappa$ with no cofinal branch.
\end{proposition}

\begin{proof}
	Let $\bb{P} = \mathrm{Coll}(\omega, \kappa)$.

	\begin{claim}
		In $V^{\bb{P}}$, $\square_\mu$ holds.
	\end{claim}

	\begin{proof}
		In $V$, let $\vec{C} = \langle C_\alpha \mid \alpha \in A \rangle$ be a $\square^{\geq \kappa^+}_\mu$-sequence. 
    $S^{\mu^+}_{\geq \kappa^+} \subseteq A$, so, in $V^{\bb{P}}$, if $\alpha \in \mathrm{lim}(\mu^+) \setminus A$, 
    then $\cf(\alpha) = \omega$. Thus, in $V^{\bb{P}}$, we can define a $\square_\mu$-sequence $\vec{D} = \langle 
    D_\alpha \mid \alpha < \mu^+ \rangle$ by letting $D_\alpha = C_\alpha$ for all $\alpha \in A$ and, for all 
    $\alpha \in \mathrm{lim}(\mu^+) \setminus A$, letting $D_\alpha$ be an arbitrary $\omega$-sequence cofinal in $\alpha$.
	\end{proof}

	A $\square_\mu$-sequence is easily seen to be a $\square^\omega(\mu^+)$-sequence. Therefore, by Propositions 
  \ref{otprop} and \ref{todorprop}, there is a subadditive, unbounded function $d:[\mu^+]^2 \rightarrow \omega$ 
  in $V^{\bb{P}}$. Thus, by Proposition \ref{subadditiveProp}, there is, in $V^{\bb{P}}$, a $\mu^+$-system 
  $S$ of width $\omega$ with no cofinal branch. Let $\dot{S}$ be a name for such an $S$. Then the derived 
  system $D_{\bb{P}}(\dot{S})$ is, in $V$, a $\mu^+$-system of width $\kappa$ with no cofinal branch.
\end{proof}

We can use this to show that, for example, the narrow system property at $\aleph_{\omega + 1}$ can hold for narrow systems of arbitrarily high width below $\aleph_\omega$ while failing in general.

\begin{corollary}
	Suppose $\lambda$ is a supercompact cardinal and $n < \omega$. Then there is a forcing extension in which every $\aleph_{\omega+1}$-system of width $\leq \aleph_n$ has a cofinal branch but there is an $\aleph_{\omega + 1}$-system of width $\aleph_{n+1}$ with no cofinal branch.
\end{corollary}

\begin{proof}
	Let $\bb{P} = \mathrm{Coll}(\aleph_{n+1}, < \lambda)$, and let $G$ be $\bb{P}$-generic over $V$. 
  In $V[G]$, let $\bb{Q} = \bb{B}(\lambda^{+\omega}, \lambda)$, and let $H$ be $\bb{Q}$-generic over 
  $V[G]$. Note that, in $V[G*H]$, $\lambda = \aleph_{n+2}$ and $\lambda^{+\omega+1} = \aleph_{\omega + 1}$. 
  Since, in $V[G]$, $\bb{Q}$ is $\lambda$-directed closed, Theorem \ref{indestthm} implies that, in 
  $V[G*H]$, $NSP(\aleph_{n+1}, \geq \aleph_{n+2})$ holds. In particular, every $\aleph_{\omega+1}$-system 
  of width $\leq \aleph_n$ has a cofinal branch. On the other hand, $\square^{\geq \aleph_{n+2}}_{\aleph_\omega}$ 
  holds in $V[G*H]$, so, by Proposition \ref{baumgartnersquareprop}, there is an $\aleph_{\omega+1}$-system of 
  width $\aleph_{n+1}$ with no cofinal branch.
\end{proof}

\section{The Proper Forcing Axiom and narrow systems} \label{pfaSect}

In this section, we investigate the extent to which the Proper Forcing Axiom (PFA) influences narrow systems. We first recall the notion of a guessing model, introduced by Viale and Weiss in \cite{vw}.

\begin{definition} \label{guessing_def}
	Let $\theta$ be a regular cardinal, and let $M \prec H(\theta)$.
	\begin{enumerate}
		\item{Suppose $X \in M$ and $d \subseteq X$.
		\begin{enumerate}
			\item{$d$ is \emph{$M$-approximated} if, for every countable $z \in M$, $z \cap d \in M$.}
			\item{$d$ is \emph{$M$-guessed} if there is $e \in M$ such that $e \cap M = d \cap M$.}
		\end{enumerate}}
		\item{$M$ is a \emph{guessing model} if every $M$-approximated $d$ is $M$-guessed.}
		\item{If $\kappa \leq \theta$, then $\mathcal{G}_\kappa (H(\theta)) = \{M \prec H(\theta) \mid |M| < \kappa$ and $M$ is a guessing model$\}.$}
	\end{enumerate}
\end{definition}

The following is proven in \cite{vw}.

\begin{theorem} \label{vwthm}
	Suppose PFA holds. Then $\mathcal{G}_{\omega_2}(H(\theta))$ is stationary in $\mathcal{P}_{\omega_2}(H(\theta))$ for every regular $\theta \geq \omega_2$.
\end{theorem}

We use this to prove the following result.

\begin{theorem} \label{pfathm}
	Suppose PFA holds. Then $NSP(\omega_1, \geq \omega_2)$ holds.
\end{theorem}

\begin{proof}
	Suppose $\lambda \geq \omega_2$ is a regular cardinal and $S = \langle I \times \omega, \mathcal{R} \rangle$ 
  is a $\lambda$-system, with $|\mathcal{R}| \leq \omega$. As before, we may assume $I = \lambda$ and $\mathcal{R} = \{R\}$. 
  We must produce a cofinal branch through $S$.

  Let $\theta$ be a sufficiently large regular cardinal. By Theorem \ref{vwthm}, $\mathcal{G}_{\omega_2}(H(\theta))$ is 
  stationary in $\mathcal{P}_{\omega_2}(H(\theta))$, so we may find $M \in \mathcal{G}_{\omega_2}(H(\theta))$ such that $\lambda, S \in M$. 
  Let $\delta = \sup(M \cap \lambda)$.

	\begin{claim}
		$\cf(\delta) > \omega$.
	\end{claim}

	\begin{proof}
		Suppose for sake of contradiction that $\cf(\delta) = \omega$. Let $A = \{\alpha_n \mid n < \omega \}$ be such that 
    $A \subseteq M$, $A$ is cofinal in $\delta$, and, for all $n < \omega$, $\alpha_n < \alpha_{n+1}$. In particular, 
    $A \subseteq \lambda \in M$, so we are in the scope of item (1) of Definition \ref{guessing_def}. Let $z \in M$ 
    be a countable set. Then $M \models ``z \cap \lambda$ is bounded below $\lambda,"$ so $z \cap \delta$ is bounded below $\delta$. 
    This implies that $z \cap A$ is a finite set and hence a member of $M$. Thus, $A$ is $M$-approximated, so, since $M$ 
    is a guessing model, there is $B \in M$ such that $B \cap M = A \cap M = A$. Fix, in $M$, a bijection 
    $f:|B| \rightarrow B$. Since $M \prec H(\theta)$, we must have $\omega+1 \subseteq M$. Therefore, 
    $f``\omega \in M$. In particular, $\sup(f``\omega) \in M$. However, as $f``\omega$ is an infinite subset of $A$, 
    we have $\sup(f``\omega) = \delta$, contradicting the fact that $\delta \not\in M$.
  \end{proof}

	For each $n < \omega$, let $d_n = \{u \in S \cap M \mid u <_R (\delta, n)\}$. $d_n$ 
  is then a branch of $S$ through $R$. For every $\alpha \in M \cap \delta$, by clause (4) of Definition \ref{system_def}, there is $n < \omega$
  such that $d_n \cap S_\alpha \not= \emptyset$. Thus, since $\cf(\delta) > \omega$, there is $n^* < \omega$ 
  such that $d_{n^*} \cap S_\alpha \not= \emptyset$ for cofinally many $\alpha \in M \cap \delta$. 
  Fix such an $n^*$, and let $d = d_{n^*}$.

	\begin{claim}
		$d$ is $M$-approximated.
	\end{claim}

	\begin{proof}
		Note first that $d \subseteq \lambda \times \omega \in M$, so we are indeed in the scope of item (1) of Definition 
    \ref{guessing_def}. Suppose $z \in M$ is countable. Then $a_z := \{\alpha \mid d \cap z \cap S_\alpha \neq \emptyset \}$ is bounded 
    below $\delta$. Let $\beta \in M \cap \delta$ be such that $a_z \subseteq \beta$ and $d \cap S_\beta \not= \emptyset$. 
    Let $d \cap S_\beta = \{v\}$. Note that, since $\beta \in M$ and $v$ is of the form $(\beta,n)$ for some $n < \omega$, 
    we have $v \in M$. Then $z \cap d = \{u \in S \cap z \mid u <_R v \}$. Everything used to define this 
    set is in $M$, so $z \cap d \in M$. 
	\end{proof}

	Since $M$ is a guessing model and $d$ is $M$-approximated, there is $b \in M$ such that $b \cap M = d \cap M$. But then $M \models ``b$ is a cofinal branch through $S,"$ so, by elementarity, $b$ is in fact a cofinal branch through $S$.
\end{proof}

The following result shows that Theorem \ref{pfathm} is sharp.

\begin{theorem}
	PFA does not imply $NSP(\omega_2, \mu^+)$ for any $\mu \geq \omega_2$.
\end{theorem}

\begin{proof}
	Suppose $\kappa$ is supercompact, and let $\mu \geq \kappa$. Assume that the supercompactness of $\kappa$ 
  is indestructible under $\kappa$-directed closed forcing. Let $\bb{B} = \bb{B}(\mu, \kappa)$ be the $\kappa$-directed 
  closed forcing to add a $\square^{\geq \kappa}_\mu$-sequence. By examining the definition of $\square^{\geq \kappa}_\mu$, 
  it is easily seen that a $\square^{\geq \kappa}_\mu$-sequence remains a $\square^{\geq \kappa}_\mu$-sequence in 
  any further extension which preserves cofinalities and cardinalities $\geq \kappa$. In $V^{\bb{B}}$, $\kappa$ remains supercompact, 
  so there is a poset $\bb{P}$ such that $\bb{P}$ preserves all cofinalities and cardinalities $\geq \kappa$ and $\Vdash_{\bb{P}}``\mathrm{PFA}$ and 
  $\kappa = \aleph_2"$. Then, in $V^{\bb{B}*\dot{\bb{P}}}$, PFA holds, but also $\square^{\geq \kappa}_\mu (= \square^{\geq \aleph_2}_\mu)$ holds, 
  which, by Proposition \ref{baumgartnersquareprop}, implies the existence of a $\mu^+$-system of width $\aleph_1$ 
  with no cofinal branch and hence the failure of $NSP(\omega_2, \mu^+)$. 
\end{proof}

\section{Open questions} \label{questionSect}

In the final section we collect a few as-yet-unanswered questions.

\begin{question} \label{tpnotnsp}
	Is it consistent that there is a singular cardinal $\mu$ such that the tree property holds at $\mu^+$ but $NSP(\mu^+)$ fails?
\end{question}

In all known models for the tree property at the successor of a singular cardinal, the narrow system property holds and is in fact a key component in the verification of the tree property, so a positive answer to Question \ref{tpnotnsp} would seem to require some new ideas.

\begin{question}
	Suppose $\mu$ is a singular cardinal and $\square^*_\mu$ holds. Must there be a $\mu^+$-tree that does not admit a narrow system?
\end{question}

First note that a cofinal branch through a $\mu^+$-tree is itself a narrow system (of width $1$), so a $\mu^+$-tree that does not admit 
a narrow system must be a $\mu^+$-Aronszajn tree. Also, as mentioned before, in models in which $\square^*_\mu$ and $NSP(\mu^+)$ both hold, 
there are $\mu^+$-Aronszajn trees and all such trees do not admit narrow systems. Finally, by a result of Magidor and Shelah, if $\cf(\mu) < \kappa < \mu$ and $\kappa$ 
is strongly compact, then every $\mu^+$-tree admits a narrow system (this is essentially the content of ``Step one'' of the proof of Theorem 
3.1 of \cite{magidorshelah}).

Consideration of the following questions about the tree property led to the results in this paper. They remain unanswered.

\begin{question} \label{tprefl}
	Is it consistent that the tree property holds at $\aleph_{\omega + 1}$ and every stationary subset of $\aleph_{\omega + 1}$ reflects?
\end{question}

Fontanella and Magidor prove in \cite{fontanella} that it is consistent that the tree property holds at $\aleph_{\omega^2+1}$ and 
every stationary subset of $\aleph_{\omega^2+1}$ reflects. In their model, $AP_{\aleph_{\omega^2}}$ fails. If $\aleph_\omega$ is 
strong limit and every stationary subset of $\aleph_{\omega + 1}$ reflects, then $AP_{\aleph_\omega}$ holds (this 
is due to Shelah; for a proof, see Corollary 3.41 of \cite{eisworth}). 
Thus, the following may be relevant in answering Question \ref{tprefl}.

\begin{question}
	Is it consistent that there is a singular cardinal $\mu$ such that $AP_\mu$ and the tree property at $\mu^+$ hold simultaneously?
\end{question}

We do not even know the situation in the following seemingly simple model.

\begin{question}
	Suppose $\langle \kappa_n \mid n < \omega \rangle$ is an increasing sequence of supercompact cardinals and $\mu = \sup(\{\kappa_n \mid n < \omega\})$. Let $\bb{P}$ be the forcing poset to shoot a club in $\mu^+$ through the set of approachable points. In $V^{\bb{P}}$, does the tree property hold at $\mu^+$?
\end{question}

\bibliography{systems}
\bibliographystyle{amsplain}

\end{document}